\documentclass{amsart}
\usepackage{amsfonts,amssymb,amscd,amsmath,enumerate,verbatim,calc}
\usepackage{xy}
\usepackage{mathtools}

\newcommand{\CM}{Cohen-Macaulay}

\newcommand{\m}{\mathfrak{m} }

\newcommand{\rt}{\rightarrow}

\newcommand{\sub}{\subseteq}

\newcommand{\depth}{\operatorname{depth}}

\theoremstyle{plain}

\newtheorem{theorem}{Theorem}[section]
\newtheorem{corollary}[theorem]{Corollary}
\newtheorem{lemma}[theorem]{Lemma}
\newtheorem{proposition}[theorem]{Proposition}

\theoremstyle{definition}
\newtheorem{definition}[theorem]{Definition}
\newtheorem{s}[theorem]{}
\newtheorem{remark}[theorem]{Remark}

\theoremstyle{remark}

\begin{document}
	\title[Associated graded modules]{On associated graded modules of maximal Cohen-Macaulay modules over hypersurface rings }
	\author{Ankit Mishra}
	\email{ankitmishra@math.iitb.ac.in}
	
	\author{ Tony~J.~Puthenpurakal}
	\email{tputhen@math.iitb.ac.in}
	
	\address{Department of Mathematics, IIT Bombay, Powai, Mumbai 400 076}

	\date{\today}
	\subjclass{Primary 13A30; Secondary  13D40, 13C15,13H10}
	
	\keywords{maximal Cohen-Macaulay module, reduction number, Ratliff-Rush filtration, associated graded module, hypersurface ring}
	
	\begin{abstract}
		Let $(A,\mathfrak{m})$ be a hypersurface ring with dimension $d$, and $M$ a MCM $A-$module with red$(M)\leq 2$ and $\mu(M)=2$ or $3$ then we have proved that depth $G(M)\geq d-\mu(M)+1$. If $e(A)=3$ and $\mu(M)=4$ then in this case  we have  proved that depth$G(M)\geq d-3$. Next we consider the case when $e(M)=\mu(M)i(M)+1$ and prove that depth $G(M)\geq d-1$. When $A = Q/(f)$ where $Q = k[[X_1,\cdots, X_{d+1}]]$ then we give estimates for $\depth G(M)$ in terms of a minimal presentation of $M$. Our paper is the first systematic study of depth of associated graded modules of MCM modules over hypersurface rings

	\end{abstract}
	\maketitle
	\section{Introduction}
	
	Let $(A,\mathfrak{m})$ be Noetherian local ring of dimension $d$ and $M$ a finite \CM\ $A$-module of dimension $r$. Let $G(A)=\bigoplus_{n\geq0}\mathfrak{m}^n/\mathfrak{m}^{n+1}$ be associated graded ring of $A$ with respect to $\mathfrak{m}$ and $G(M)=\bigoplus_{n\geq0}{\mathfrak{m}^nM}/{\mathfrak{m}^{n+1}M}$ be associated graded module of $M$ with respect to $\mathfrak{m}$. Now $\mathcal{M}=\bigoplus_{n\geq1}\mathfrak{m}^n/\mathfrak{m}^{n+1}$ is irrelevant maximal ideal of $G(A)$ we set depth $G(M)$=grade$(\mathcal{M},G(M))$. If $L$ be an $A$-module then minimal number of generators of $L$ is denoted by  $\mu(L)$ and its length is denoted by $\ell(L)$.
	
	We know that Hilbert-Samuel function of $M$ with respect to $\mathfrak{m}$ is
	$$H^1(M,n)=\ell({M}/{\mathfrak{m}^{n+1}M})\ \text{for all}\ n\geq0.$$
	There exists a polynomial $P_M(z)$ of degree $r$  such that
	  $$H^1(M,n)=P_M(n)\ \text{for}\ n\gg0.$$
	This polynomial can be written as $$P_M(X)=\sum_{i=0}^{r}(-1)^ie_i(M)\binom{X+r-i}{r-i}$$
	These coefficients $e_i(M)'$s are integers and known as {\it Hilbert coefficients} of $M$.
	
	We know that Hilbert series of $M$ is formal power series $$H_M(z)=\sum_{n\geq0}\ell(\mathfrak{m}^nM/\mathfrak{m}^{n+1}M)z^n$$
	We can write $$ H_M(z)=\frac{h_M(z)}{(1-z)^r},\ \text{where}\ r=dimM $$
	
	Here, $h_M(z)=h_0(M)+h_1(M)z+\ldots+h_s(M)z^s\in \mathbb{Z}[z]$ and $h_M(1)\neq0$. This polynomial is know as {\it h-polynomial} of $M$.

	If we set $f^{(i)}$ to denote $i$th formal derivative of  a polynomial $f$ then it is easy to see that $e_i(M)=h_M^{(i)}(1)/i!$ for $i=0,\ldots,r$.  It is  convenient to set $e_i(M)=h_M^{(i)}(1)/i!$ for all $i\geq0.$
	
	Now we know that if $(A,\mathfrak{m})$ is \CM\ with red$(A)\leq2$ then $G(A)$ is \CM\ (see\cite[Theorem 2.1]{S}). 
	
	If $M$ is a \CM \ $A$-module with red$(M)\leq1$ then  $G(M)$ is \CM \ (see \cite[Theorem 16]{Pu0}), but if red$(M)=2$, then $G(M)$ need not be \CM\ (see \cite[Example 3.3]{PuMCM}).
	
	Here  we consider maximal \CM\ (MCM) modules over a \CM\ local ring $(A,\mathfrak{m})$. We know that if $A$ is a regular local ring then $M$ is free, say $M\cong A^s$. This implies $G(M)\cong G(A)^s$ is \CM.
	
	The next case is when $A$ is a hypersurface ring. {\it For convenience we assume $A=Q/(f)$ where $(Q,\mathfrak{n})$ is a regular local ring with infinite residue field and $f\in \mathfrak{n}^2$.} 
	
	If $f\in \mathfrak{n}^2\setminus\mathfrak{n}^3$ then $A$ has minimal multiplicity. It follows that any MCM module $M$ over $A$ has minimal multiplicity. So $G(M)$ is \CM.
	
	One of the cases of interest for us was when  $f\in \mathfrak{n}^3\setminus\mathfrak{n}^4$. Note in this case red$(A)=2$. So if $M$ is any MCM $A$-module then red$(M)\leq2$. In this case $G(M)$ need not \CM\ (see  \cite[Example 3.3]{PuMCM}).
	
	So we study the case when red$(M)\leq 2$ for an MCM $A$-module where $f\in \mathfrak{n}^e\setminus\mathfrak{n}^{e+1}$ and $e\geq 3$.

	Notice if $M$ is an MCM module over $A$ then projdim$_Q(M)=1$. So, $M$ has a minimal presentation over $Q$ 
	$$0\rt Q^{\mu(M)}\xrightarrow{\phi}Q^{\mu(M)}\rt M\rt 0.$$
	We investigate $G(M)$ in terms of invariants of a minimal presentation of $M$ over $Q$.

	Now the first theorem which we have proved is for $\mu(M)=2$
		\begin{theorem}\label{1}
		Let $(A,\mathfrak{m})$ be a hypersurface ring  of dimension $d$ with infinite residue field and  $M$ an MCM $A$-module with $\mu(M)=2$. If
		 $red(M)\leq 2$ then depth$G(M)\geq d-1$.
	\end{theorem}

Next theorem deals with the case when $\mu(M)=3 $ and we have proved that

\begin{theorem}\label{2}
	Let $(A,\mathfrak{m})$ be a hypersurface ring  of dimension $d$ with infinite residue field and  $M$ an MCM $A$-module with $\mu(M)=3$.
	If  $red(M) \leq 2$, then depth$G(M)\geq d-2$.
\end{theorem}
Next theorem deals with the case $\mu(M)=4$ and we have proved that
\begin{theorem}\label{3}
	Let $(A,\mathfrak{m})$ be a  hypersurface ring of dimension $d$ with $e(A)=3$ and $M$ an MCM $A$-module. If   $\mu(M)=4$, then depth$G(M)\geq d-3$.
\end{theorem}

 If $\mu(M)=r$ and $det(\phi) \in \mathfrak{n}^r\setminus\mathfrak{n}^{r+1}$  then we know that $e(M)=\mu(M)$ (see \cite[Theorem 2]{PuMCM}). So $M$ is an Ulrich module. This implies $G(M)$ is \CM. Here we consider the case when $det(\phi) \in \mathfrak{n}^{r+1}\setminus\mathfrak{n}^{r+2}$ and  prove 
\begin{theorem}\label{4}
	Let $({Q},\mathfrak{n})$ be regular local ring of dimension $d+1$ with $d\geq 0$. Let $M$ be a $Q$-module with minimal presentation $$0\rt Q^r\xrightarrow{\phi} Q^r \rt M \rt 0$$
	
	Now if $\phi = [a_{ij}]
	$
	where $a_{ij} \in \mathfrak{n}$ with  $f=det(\phi) \in \mathfrak{n}^{r+1}\setminus \mathfrak{n}^{r+2}$ and $red(M)\leq 2$, then depth$G(M)\geq d-1$. In this case we can also prove that
	\begin{enumerate}
		\item $G(M)$ is \CM \  if and only if $h_M(z)=r+z$.
		\item depth $G(M)=d-1$ if and only if $h_M(z)=r+z^2$
	\end{enumerate}
\end{theorem}

Let $0\rt Q^{\mu(M)}\xrightarrow{\phi}Q^{\mu(M)}\rt M\rt 0$ be a minimal presentation of $M$ over $Q$. Set $i(M)=$ max\{$i|$all entries of  $\phi$ are in $\mathfrak{n}^i$\}. Then
from \cite[Theorem 2]{PuMCM} we know that $e(M)\geq \mu(M)i(M)$ for any MCM module $M$ over a hypersurface ring; in that paper, it is given that if $e(M)=\mu(M)i(M)$ then $G(M)$ is \CM. 

Here we  consider the case when   $e(M)=\mu(M)i(M)+1$ and prove that:
\begin{theorem}\label{5}
	Let $(Q,\mathfrak{n})$ be a regular local ring of dimension $d+1$, $g\in \mathfrak{n}^i\setminus\mathfrak{n}^{i+1}$, $i\geq 2$. Let $(A,\mathfrak{m})=(Q/(g),\mathfrak{n}/(g))$  and $M$ be an MCM $A$-module. Now  if $e(M)=\mu(M)i(M)+1$ then depth$G(M)\geq d-1$ and $h_M(z)=\mu(M)(1+z+\ldots+z^{i(M)-1})+z^s$ where $s\geq i(M)$. Furthermore, $G(M)$ is \CM\ if and only if $s=i(M)$.  
\end{theorem}

Now if we set $Q=k[[x_1,\ldots,x_{d+1}]]$ and $\phi=\sum_{i\geq i(M)}\phi_i$, where  $\phi_i$'s are forms of degree $i$.
We know that if $det\phi_{i(M)}\neq0$ then $G(M)$ is \CM\ (see \cite[Proposition 4.1]{PuMCM}). Here we consider the case when rank($\phi_{i(M)})=\mu(M)-1$ with det$(\phi_{i(M)}+\phi_{i(M)+1})\ne0$ and prove that

\begin{corollary} \label{6}
	Let $Q=k[[x_1,\ldots,x_{d+1}]]$.  Let $M$ be $Q$-modules with minimal presentation $0\rt Q^{\mu(M)}\xrightarrow{\phi}Q^{\mu(M)}\rt M\rt 0$. Set $\phi=\sum_{i\geq i(M)}\phi_i$, where $\phi_i$'s are forms of degree $i$. Now if rank($\phi_{i(M)})=\mu(M)-1$ and det$(\phi_{i(M)}+\phi_{i(M)+1})\ne0$ then depth$G(M)\geq d-1$ and
	$h_M(z)=\mu(M)(1+z+\ldots+z^{i(M)-1})+z^s$ where $s\geq i(M)$. Also, $G(M)$ is \CM\ if and only if $s=i(M)$.
	
\end{corollary}

This result ( for $\mu(M)=2,3,4$ and dim$M=2,3$ ) was guessed after many numerical computations done in 2004 by Sangeeta Maini, a project student of the second author.

Here is an overview of the contents of this paper. In  section 2, we give some preliminary which we have used in the paper. In section 3, we discuss $\mu(M)=2$ case and prove Theorem \ref{1}. In section 4, we discuss $\mu(M)=3$ case and prove Theorem \ref{2}. In section 5, we discuss $\mu(M)=4$ case and prove Theorem \ref{3}. In section 6, we discuss $\mu(M)=r$ case and prove theorem \ref{4}. In section 7, we prove Theorem \ref{5}, and as its corollary we prove Corollary \ref{6}. In the last section examples are given.

\section{Priliminaries}
Let $(A,\mathfrak{m})$ be a Noetherian local ring of dimension $d$, and $M$ an $A$-module of dimension $r$.
\s An element $x\in \mathfrak{m}$ is said to be a superficial element of $M$ if there exists an integer $n_0>0$ such that $$(\mathfrak{m}^nM:_Mx)\cap \mathfrak{m}^{n_0}M=\mathfrak{m}^{n-1}M\ \text{for all}\ n>n_0$$
We know that if residue field $k=A/\mathfrak{m}$ is infinite then superficial elements always exist (see \cite[Pg 7]{Sbook}). A sequence of elements $x_1,\ldots,x_m$ is said to be superficial sequence if $x_1$ is $M$-superficial and $x_i$ is $M/(x_1,\ldots,x_{i-1})M$-superficial for $i=2,\ldots,m.$

\begin{remark}
	
	\begin{enumerate}
		\item 	If $x$ is $M-$superficial and regular then we have $(\mathfrak{m}^nM:_Mx)=\mathfrak{m}^{n-1}M$ for all $n\gg 0.$
		
		\item If depth$M>0$ then it is easy to show that  every $M$-superficial element is also $M-$ regular.
	\end{enumerate}

\end{remark}

\s \label{Base change} Let $f:(A,\mathfrak{m})\rt (B,\mathfrak{n})$ be a flat local  ring homomorphism with $\mathfrak{m}B=\mathfrak{n}$. If $M$ is an $A$-module set $M'=M\otimes_A B$, then  following facts are well known 
\begin{enumerate}
	\item $H(M,n)=H(M',n)$ for all $n\geq0$.
	\item depth$_{G(A)}G(M)=$depth$_{G(A')}G(M')$.
	\item projdim$_AM$=projdim$_{A'}M'$
	
\end{enumerate}
We will use this result in the following two cases:
\begin{enumerate}
	\item We can assume $A$ is complete by taking $B=\hat{A}$.
	\item We can assume the residue field of $A$ is infinite, because if the residue field $(k=A/\mathfrak{m})$ is finite we can take $B=A[X]_S$  where $S=A[x]\setminus \mathfrak{m}A[X]$. Clearly, the residue field of $B=k(X)$ is infinite.
\end{enumerate}

Since all the properties we deal in this article are invariant when we go from $A$ to $A'$. So we can assume that residue field of $A$ is infinite.

\s If $a$ is a non-zero element of $M$ and if $i$ is the largest integer such that $a\in \mathfrak{m}^iM$, then we denote image of $a$ in $\mathfrak{m}^i \  M/\mathfrak{m}^{i+1} \ M$ by $a^*$. If $N$ is a submodule of $M$, then $N^*$ denotes the graded submodule of $G(M)$ generated by all $b^*$ with $b\in N$.

\begin{definition}
	Let $(A,\mathfrak{m})$ be a Noetherian local ring and $M\ne 0$ be a finite $A$-module then $M$ is said to be a \CM\ $A$-module if depth $M=$dim $M$, and a maximal \CM\ (MCM) module if depth $M=$dim $A$.
\end{definition}

\s \label{mod-sup} If $x\in \mathfrak{m}\setminus\mathfrak{m}^2$ an $M-$superficial and regular element. Set $N=M/xM$ and $K=\mathfrak{m}/(x)$ then we have {\bf Singh's equality}\index{Singh's equality} ( for $M=A$ see \cite[Theorem 1]{singh}, and for the module case see \cite[Theorem 9]{Pu0})
$$H(M,n)=\ell(N/K^{n+1}N)-\ell\left(\frac{\mathfrak{m}^{n+1}M:x}{\mathfrak{m}^nM}\right)\ \text{for all}\ n\geq0.$$

Set $b_n(x,M)=\ell(\mathfrak{m}^{n+1}M:x/\mathfrak{m}^nM)$ and $b_{x,M}(z)=\sum_{n\geq0}b_n(x,M)z^n$. Notice that $b_0(x,M)=0$. Now we have 
$$h_M(z)=h_N(z)-(1-z)^rb_{x,M}(z)$$

\s \label{Property} (See \cite[Corollary 10]{Pu0}) Let $x\in \mathfrak{m}$ be an $M-$superficial and regular element. Set $B=A/(x)$, $N=M/xM$ and $K=\mathfrak{m}/(x)$ then we have
\begin{enumerate}
	\item dim$M-1$ = dim$N$ and $h_0(N)=h_0(M)$.
	\item $b_{x,M}$ is a polynomial.
	\item $h_1(M)=h_1(N)$ if and only if $\mathfrak{m}^2M\cap xM=x\mathfrak{m}M.$
	\item $e_i(M)=e_i(N)$ for $i=0,\ldots,r-1.$
	\item $e_r(M)=e_r(N)-(-1)^r\sum_{n\geq0}b_n(x,M).$
	\item $x^*$ is $G(M)$-regular if and only if $b_n(x,M)=0$ for all $n\geq0.$
	\item $e_r(M)=e_r(N)$ if and only if $x^*$ is $G(M)$-regular.
	\item depth $G(M)\geq1$ if and only if $h_{M}(z)=h_N(z)$.
\end{enumerate}
\s \label{Sally-des} {\bf Sally-descent }(see \cite[Theorem 8]{Pu0}): Let $(A,\mathfrak{m})$ be a \CM\ local ring of dimension $d$ and $M$ be \CM\ module of dimension $r$. Let $x_1,\ldots,x_c$ be a $M$-superficial sequence with $c\leq r$. Set $N=M/(x_1,\ldots,x_c)M$ then   

depth $G(M)\geq c+1$ if and only if depth $G(N)\geq 1$. 
\s The reduction number\index{reduction number} of $M$ can be defined as the least integer $\ell$ such that there is an ideal $J$ generated by a maximal superficial sequence with $\mathfrak{m}^{\ell+1}M=J\mathfrak{m}^{\ell}M$.

\begin{definition}\label{Ulrich}
	Let $(A,\mathfrak{m})$ be a Noetherian local ring and $M$ be a maximal \CM \ module  then $M$ is said to be a Ulrich module if $e(M)=\mu(M)$.

\end{definition}
\begin{remark}
	When $M$ is an MCM module and $\mathfrak{m}$ has a minimal reduction $J$ generated by a system of parameters, then $M$ is Ulrich module if and only if $\mathfrak{m}M=JM$. 
\end{remark}

\s  (See \cite[section 6]{heinzer}) For any $n\geq1$ we can define Ratliff-Rush submodule of $M$ associated with $\mathfrak{m}^n$ as 
$$\widetilde{\mathfrak{m}^nM}=\bigcup_{i\geq0}(\mathfrak{m}^{n+i}M:_M\mathfrak{m}^i)$$ The filtration $\{\widetilde{\mathfrak{m}^nM}\}_{n\geq1}$ is known as the Ratliff-Rush filtration \index{Ratliff-Rush filtration} of $M$ with respect to $\mathfrak{m}$.

For the proof of the following properties in the ring case see \cite{Ratliff}. This proof can be easily extended for the modules. Also see \cite[2.2]{Naghipour}.

\s If depth$(M)>0$ and  $x\in \mathfrak{m}$ is a $M-$superficial element then we have
\begin{enumerate}
	\item $\widetilde{\mathfrak{m}^nM}=\mathfrak{m}^nM$ for all $n\gg0.$
	\item  $(\widetilde{\mathfrak{m}^{n+1}M}:x)=\widetilde{\mathfrak{m}^nM}$ for all $n\geq1.$
\end{enumerate}
\s\label{htilde} Let $\widetilde{G(M)}=\bigoplus_{n\geq0}\widetilde{\mathfrak{m}^nM}/\widetilde{\mathfrak{m}^{n+1}M}$ be the associated graded module  of $M$ with respect to Ratliff-Rush filtration. Then its Hilbert series
$$\sum_{n\geq0}\ell(\widetilde{\mathfrak{m}^nM}/\widetilde{\mathfrak{m}^{n+1}M})z^n=\frac{\widetilde{h_M}(z)}{(1-z)^r}$$
Where $\widetilde{h_M}(z)\in \mathbb{Z}[z]$. Set $r_M(z)=\sum_{n\geq0}\ell(\widetilde{\mathfrak{m}^{n+1}M}/\mathfrak{m}^{n+1}M)z^n$; clearly, $r_M(z)$ is a polynomial with non-negative integer coefficients (because depth$M>0$). Now we have $$h_M(z)=\widetilde{h_M}(z)+(1-z)^{r+1}r_M(z);\ \text{where }\ r=\text{dim}M$$
We know that depth$G(M)>0$ if and only if $r_M(z)=0$.

\s (see \cite[2.1]{Pu2}) Let $x$ be an $M-$superficial element and depth$M\geq2$. Set $N=M/xM$, then we have a natural map $\rho^x:M\rt N$ and we say that Ratliff-Rush filtration on $M$ behaves well mod superficial element $x$ if $\rho^x(\widetilde{\mathfrak{m}^nM})=\widetilde{\mathfrak{m}^nN}$ for all $n\geq 1$. Now $\rho^x$ induces the maps $$\rho_n^x:\frac{\widetilde{\mathfrak{m}^nM}}{\mathfrak{m}^nM}\rt \frac{\widetilde{\mathfrak{m}^nN}}{\mathfrak{m}^nN}$$
It is easy to  show that Ratliff-Rush filtration behaves well mod $x$ if and only if $\rho_n^x$ is surjective for all $n\geq 1$.

\begin{definition}
	Let $(A,\mathfrak{m})$ be a Noetherian local ring and $M$ be a finite $A$-module with dim$M=d$. Then we say $G(M)$ is a generalized \CM \ module if 
	$$\ell(H^i_\mathcal{M}(G(M)))<\infty\ \text{ for } \ i=0,\ldots,d-1$$
	where, $H^i_\mathcal{M}(G(M))$ is the $i$-th local cohomology module of $G(M)$ with respect to the maximal homogeneous ideal $\mathcal{M}$. 
\end{definition}
\begin{remark}
	$G(A)$ is a finitely generated $k(=A/\mathfrak{m})$-algebra. A $G(A)$-module $E$ is generalized \CM \ if and only if $E_P$ is \CM\ for all prime ideals $P\ne \mathcal{M}$.
\end{remark}

\begin{proposition}\label{ASSG}
	Let $(A,\mathfrak{m})$ be a  \CM\ local ring of dimension $d\geq 1$ and $M$ a finite $A$-module with dim$M=d$. Now if  $\widetilde{G(M)}$ is a \CM \ $G(A)$-module, then 
	\begin{enumerate}
		\item $G(M)$ is a generalized \CM\ module.
		\item dim $G(A)/P = d$ for all minimal primes $P$ of $G(M)$.
	\end{enumerate}
	
\end{proposition}

\begin{definition}\label{hypersurface}
	Let $(A,\mathfrak{m})$ be a Noetherian local ring, then $A$ is said to be a hypersurface ring if its completion can be written as a quotient of a regular local ring by a principal ideal.
\end{definition}
\s Let $(Q,\mathfrak{n})$ be a regular local ring, $f\in \mathfrak{n}^e\setminus\mathfrak{n}^{e+1}$ and $A=Q/(f)$. If $M $ is an MCM $A-$module then projdim$_Q(M)=1$ and $M$ has a minimal presentation: $$0\rt Q^{\mu(M)}\rt Q^{\mu(M)}\rt M \rt 0$$

\s \label{i(M)}Let $(Q,\mathfrak{n})$ be a regular local ring and $\phi : Q^t\rt Q^t$  a linear map,  set
$$i_\phi=\text{max}\{i |\  \text{all entries of}\ \phi \ \text{are in }\ \mathfrak{n}^i\}$$
If $M $ has minimal presentations: $0\rt Q^t\xrightarrow{\phi}Q^t\rt M\rt0$ and
$0\rt Q^t\xrightarrow{\phi'}Q^t\rt M\rt0$ then it is well known that $i_\phi=i_{\phi'}$ and det$\phi=u$det$\phi'$ where $u$ is a unit. We set $i(M)=i_\phi$ and det$M=$det$\phi$. For any non-zero element $a$ of $Q$ we set $v_Q(a)=max\{i|a\in \mathfrak{n}^i\}$. We are choosing this set-up from \cite{PuMCM}.

\begin{definition} \label{phi}
	(See \cite[Definition 4.4]{PuMCM}) Let $(Q,\mathfrak{n})$ be a regular local ring, $A=Q/(f)$ where $f\in \mathfrak{n}^e\setminus\mathfrak{n}^{e+1}, e\geq2$ and $M$  an MCM $A-$module with minimal presentation: $$0\rt Q^t\xrightarrow{\phi}Q^t\rt M\rt0$$
	Then an element $x$ of $\mathfrak{n}$ is said to be $\phi-$ superficial \index{$\phi-$ superficial}if we have
	\begin{enumerate}
		\item $x$ is $Q\oplus A\oplus M$ superficial.
		\item If $\phi=(\phi_{ij})$ then $v_Q(\phi_{ij})=v_{Q/xQ}(\overline{\phi_{ij}})$.
		\item $v_Q(det(\phi))=v_{Q/xQ}det(\overline{\phi})$
	\end{enumerate}
\end{definition}
\begin{remark}

	If $x$ is $Q\oplus A\oplus M\oplus (\oplus_{ij}Q/(\phi_{ij}))\oplus Q/(det(\phi))-$superficial then it is $\phi-$superficial. So if the residue field of $Q$ is infinite then $\phi-$superficial elements always exist.
\end{remark}
\begin{definition}
	(See \cite[Definition 4.5]{PuMCM}) Let $(Q,\mathfrak{n})$ be a regular local ring, $A=Q/(f)$ where $f\in \mathfrak{n}^e\setminus\mathfrak{n}^{e+1}, e\geq2$ and $M$  an MCM $A-$module with minimal presentation: $$0\rt Q^t\xrightarrow{\phi}Q^t\rt M\rt0.$$
	We say that $x_1,\ldots,x_c$ is a $\phi$-superficial sequence if $\overline{x_n}$ is $(\phi \otimes_Q Q/(x_1,\ldots,x_{n-1}))$-superficial for $n=1,\ldots,c$.
\end{definition}

\begin{remark}
	Assume residue field of $Q$ is infinite. If red$_{J_0}(M)\leq 2$ for one minimal reduction $J_0$ of $M$. Then for almost all minimal reduction $J$ of $M$, red$_J(M)\leq 2$. { \it We will use this fact implicitly.}
\end{remark}
\s \label{d=1} With above set-up we have   
\begin{enumerate}
	\item (see \cite[Lemma 4.7]{PuMCM})  If dim$M=1$  then
	$$h_M(z)=\mu(M)(1+z+\ldots+z^{i(M)-1})+\sum_{i\geq i(M)}h_i(M)z^i$$
	$$ \text{with}\ h_i(M)\geq0 \ \forall \ i.$$
	
	\item (see \cite[Theorem 2]{PuMCM}) $e(M)\geq \mu(M)i(M)$ and if $e(M)=\mu(M)i(M)$ then 
	
	$G(M)$ is Cohen-Macaulay and $h_M(z)=\mu(M)(1+z+\ldots+z^{i(M)-1})$.
	
\end{enumerate}

\s \label{exact seq}Let $(A,\mathfrak{m})$ be a \CM\ local ring  and $M$ a \CM\ $A$-module of dimension 2. Let $x,y$ be a maximal $M$-superficial sequence. 

Set $J=(x,y)$ and $\overline{M}=M/xM$ then we have exact sequence (for $M=A$ see \cite[Lemma 2.2]{rv})
\begin{align*}
0 \rt \mathfrak{m}^nM:J/\mathfrak{m}^{n-1}M\xrightarrow{f_1} \mathfrak{m}^nM:x/\mathfrak{m}^{n-1}M & \xrightarrow{f_2} \mathfrak{m}^{n+1}M:x/\mathfrak{m}^{n}M\\
\xrightarrow{f_3} \mathfrak{m}^{n+1}M/J\mathfrak{m}^nM
&\xrightarrow{f_4} \mathfrak{m}^{n+1}\overline{M}/y\mathfrak{m}^n\overline{M}\rt 0
\end{align*}
Here, $f_1$ is inclusion map, $f_2(a+\mathfrak{m}^{n-1}M)=ay+\mathfrak{m}^nM, f_3(b+\mathfrak{m}^nM)=bx+J\mathfrak{m}^nM$ and $f_4$ is reduction modulo $x$.

\s \label{exact d}Let $(A,\mathfrak{m})$ be a \CM\ local ring of dimension $d\geq 1$ and $M$ a maximal \CM\ $A$-module. Let $\underline{x}=x_1,\ldots,x_d$ be a maximal $M$-superficial sequence. Set $N=M/x_1M$, $J=(x_1,\ldots,x_d)$ and $\overline{J}$ is image of $J$ is $A/(x_1)$. Then we have  
$$0 \rt \mathfrak{m}^{2}M:x_1/\mathfrak{m}M
\xrightarrow{f} \mathfrak{m}^{2}M/J\mathfrak{m}M
\xrightarrow{g} \mathfrak{m}^{2}{N}/\overline{J}\mathfrak{m}{N}\rt 0.$$
Here, $f(a+\mathfrak{m}M)=ax_1+J\mathfrak{m}M$ and $g$ is reduction modulo $x_1$.

\s\label{exact1} Let $(A,\mathfrak{m})$ be a \CM\ local ring of dimension one and $M$ a maximal \CM\ $A$-module. Let $x$ be a superficial element of $M$. Set $N=M/xM$. Then we have $$0\rt \mathfrak{m}^2M:x/\mathfrak{m}^2M\xrightarrow{f} \mathfrak{m}^2M/x\mathfrak{m}^2M\xrightarrow{g}\mathfrak{m}^2N/0\rt 0.$$
Here, $f(a+\mathfrak{m}^2M)=ax+x\mathfrak{m}^2M$ and $g$ is reduction modulo $x$.

The following  result is well known, but we will use this many times. For the convenience of the reader we state it

\s \label{overline{G(M)}} Let $(Q,\mathfrak{n},k)$ be a regular local ring of dimension $d+1$ and\\ $(A,\mathfrak{m})=(Q/(f),\mathfrak{n}/(f))$ where $f\in \mathfrak{n}^i\setminus\mathfrak{n}^{i+1}$. Now if $M$ is a maximal \CM\ $A$-module with red$M\leq2$. Let $\underline{x}=x_1,\ldots,x_d$ be sufficiently general linear forms in $\mathfrak{n}/\mathfrak{n}^2$. Set $S=G_{\mathfrak{n}}(Q)$, $R=S/(\underline{x^*})S$ then $R\cong k[T]$ and $G(A)/(\underline{x})G(A)\cong R/(T^s)$ for some $s\geq2$.\\ Now consider $\overline{G(M)}=G(M)/(\underline{x^*})G(M)$. Then
$$\overline{G(M)}=M/\mathfrak{m}M \oplus \mathfrak{m}M/(\mathfrak{m}^2M+(\underline{x})M) \oplus \mathfrak{m}^2M/(\mathfrak{m}^3M+(\underline{x})\mathfrak{m}M)$$
Its Hilbert series is $\mu(M)+\alpha z+\beta z^2$ where $\beta\leq \alpha\leq \mu(M)$, because it is an $R$-module which is also $R/(T^s)$-module and it is generated in degree zero.

\s \label{RR-2} Let $(A,\mathfrak{m})$ be a \CM\ local ring of dimension $d$ and $M$ be a finite $A$-module with depth$M\geq 2$. Let $x$ be an $M$-superficial element. Set $N=M/xM$. Then for $n\geq 0$ we have exact sequence  (see \cite[2.2]{Pu2}) 
$$0\rt \frac{(\mathfrak{m}^{n+1}M:x)}{\mathfrak{m}^nM}\rt \frac{\widetilde{\mathfrak{m}^nM}}{\mathfrak{m}^nM}\rt \frac{\widetilde{\mathfrak{m}^{n+1}M}}{\mathfrak{m}^{n+1}M} \rt \frac{\widetilde{\mathfrak{m}^{n+1}N}}{\mathfrak{m}^{n+1}N}. $$

In particular, we have exact sequence
$$0\rt \widetilde{\mathfrak{m}M}/\mathfrak{m}M\rt \widetilde{\mathfrak{m}N}/\mathfrak{m}N.$$

If depth$M=1$, then for all $n\geq 0$ we have following exact sequence 
$$0\rt \frac{(\mathfrak{m}^{n+1}M:x)}{\mathfrak{m}^nM}\rt \frac{\widetilde{\mathfrak{m}^nM}}{\mathfrak{m}^nM}\rt \frac{\widetilde{\mathfrak{m}^{n+1}M}}{\mathfrak{m}^{n+1}M}$$

 The next result is a basic fact from linear algebra. But for the sake of completion we give a proof.
\begin{proposition}\label{vector} Let $V$ be a vector space of dimension $d\geq2 $ over an infinite field $k$. Let $V_1,\ldots V_n$ be finitely many proper subspaces of $V$. If dim$_k(V_i)\leq $ dim$_kV-2$, then there exists a subspace $H=ka\oplus kb$ where $a,b\in V$ such that $H\cap V_i=0$ for $i=0,\ldots, n$.
\end{proposition}
\begin{proof}
	If dim$_kV=2$ then there is nothing to prove.\\
	Now suppose dim$_kV\geq 3$. Since $k$ is infinite, $V\ne \bigcup_{i=1}^nV_i$. So there is an element $a\in V\setminus \bigcup_{i=1}^nV_i$. Now consider $$W_i=ka\oplus V_i\ \ \text{for all}\ i=1,\ldots,n.$$
	Clearly, $W_i$ for $i=1,\ldots,n$ are finitely many proper subspaces of $V$.\\ We know that $V\ne \bigcup_{i=1}^nW_i $. So there is an element $b\in V\setminus \bigcup_{i=1}^nW_i $. Therefore, $H=ka\oplus kb$ is the required subspace.
\end{proof}
{\bf Convention}:
Let $M$ be a maximal Cohen-Macaulay  module of dimension $d$ and $\underline{x}=x_1,\ldots,x_d$ be a maximal $\phi$-superficial sequence, then

$M_0=M$ and $M_t=M/(x_1,\ldots,x_t)M$ for $t=1,\ldots,d$.  

{\it We are mainly interested in red$(M)=2$ case. However, it is more convenient to prove following results for red$(M)\leq 2$.}	
	\section{\bf The case when $\mu(M)=2$ }
	\begin{theorem}\label{muM=2}
		Let $(A,\mathfrak{m})$ be a hypersurface ring  of dimension $d$ with infinite residue field and  $M$ an MCM $A$-module with $\mu(M)=2$. Now if
		$red(M)\leq2$ then depth$G(M)\geq d-1$.
	\end{theorem}
	\begin{proof}
		By \ref{Base change} we can assume that $A$ is a complete local ring.\\
		So, $(A,\mathfrak{m})=(Q/(f),\mathfrak{n}/(f))$, where
		$(Q,\mathfrak{n})$ is a regular local ring of dimension $d+1$ and $f\in \mathfrak{n}^i\setminus\mathfrak{n}^{i+1}$ for some $i\geq2$.\\
		Let dim$M\geq1$ and $0\rt Q^2\xrightarrow{\phi} Q^2\rt M\rt 0$ be a minimal presentation of $M$. Let $\underline{x}=x_1,\ldots,x_d$ be a maximal $\phi$-superficial sequence (see \ref{phi}). Set $M_d=M/\underline{x}M$ and $(Q',(y))=(Q/(\underline{x}),\mathfrak{n}/(\underline{x}))$.\\
		Clearly, $Q'$ is DVR and so $M_d\cong Q'/(y^{a_1})\oplus Q'/(y^{a_2})$. As $red(M)\leq2$, we can assume $1\leq a_1\leq a_2\leq 3$. We consider all possibilities separately 
		
		{\bf Case (1):} $a_1=a_2=1 .$ \\
		In this case  $M_d\cong Q'/(y)\oplus Q'/(y)$. This implies $h_{M_d}(z)=2$ and 
		$e(M_d)=\mu(M_d)=2$. For dim$M\geq 1$ we know that $e(M)=e(M_d)$ and $\mu(M)=\mu(M_d)$. So $e(M)=\mu(M)=2$ and this implies  $M$ is Ulrich module. Therefore $G(M)$ is \CM \ and $h_M(z)=2$ (see \cite[Theorem 2]{PuMCM}).
		
		{\bf Case (2):} $a_1=1,a_2=2 .$ \\
		In this case $M_d\cong Q'/(y)\oplus Q'/(y^2)$, so $h_{M_d}(z)=2+z$.\\
		We first consider the case when dim$M=2$ because if dim$M\leq1$ there is nothing to prove.\\ 
		Let $\underline{x}=x_1,x_2$ be a maximal $\phi$-superficial sequence. Set $J=(x_1,x_2)$ and ${M_1}=M/x_1M$.\\
		Since dim${M_1}=1$,  we can write $h$-polynomial of ${M_1}$ as
		$h_{{M_1}}(z)=2+(\rho_0-\rho_1)z+\rho_1z^2$, where  $\rho_n = \ell(\mathfrak{m}^{n+1}{M_1}/{x_2}\mathfrak{m}^n{M_1})$. Since $red(M)\leq2$, so $\rho_n=0$ for $n\geq2$. As we know that $e({M_1})=e(M_2)=3$ and $\mu({M_1})=2$ so we have $\rho_0=\ell(\mathfrak{m}{M_1}/x_2{M_1})=1$. We also know that coefficients of $h_{{M_1}}$ are non-negative [from \ref{d=1}(1)].\\ So possible values of  $\rho_1$ are $0$ and $1$.\\ 
		{\bf Subcase(i): $\rho_1=0.$}\\ Notice in this case $M_1$ has minimal multiplicity. This implies $h_{{M_1}}(z)=2+z$ and $G({M_1})$ is \CM. By Sally-descent  $G(M)$ is \CM \ and $h_M(z)=2+z$. \\
		{\bf Subcase(ii): $\rho_1=1.$ }\\
		In this case $h_{{M_1}}(z)=2+z^2$. As $h_{M_1}(z)\ne h_{M_2}(z)$, it follows that  depth$G({M_1})=0$ (see \ref{Property}).\\ As we know  [from \ref{mod-sup}] $$h_M(z)=h_{{M_1}}(z)-(1-z)^2b_{x_1,M}(z).$$
		This gives us $$e_2(M)=e_2({M_1})-\sum b_i(x_1,M)$$
		where $b_i(x_1,M)=\ell(\mathfrak{m}^{i+1}M:x_1/\mathfrak{m}^iM)$. We know that $e_2(M)$ and $\sum b_i(x_1,M) $  are non-negative integers and in this case $e_2({M_1})=1$. 
		So possible values for $\sum b_i(x_1,M) $ are $0,1$. (Also note that $b_0(x_1,M)=0$).\\
		{\bf Claim:} depth$G(M)=1$.\\
		{\bf Proof of the claim:} If possible assume that depth$G(M)=0$ then   from \ref{Property} we have $\sum b_i(x_1,M)\ne 0 $. This implies $\sum b_i(x_1,M)=1 $
		
		From the exact sequence (\ref{exact seq})
		\begin{align*}
		0 \rt \mathfrak{m}^nM:J/\mathfrak{m}^{n-1}M\rt \mathfrak{m}^nM:x_1/\mathfrak{m}^{n-1}M & \rt \mathfrak{m}^{n+1}M:x_1/\mathfrak{m}^{n}M\\
		\rt \mathfrak{m}^{n+1}M/J\mathfrak{m}^nM
		&\rt \mathfrak{m}^{n+1}{M_1}/x_2\mathfrak{m}^n{M_1}\rt 0
		\end{align*}
		
		we have  if $b_1(x_1,M)=\ell(\mathfrak{m}^2M:x_1/\mathfrak{m}M)=0$ then $b_i(x_1,M)=0$ for all $i\geq2$. \\
		We have assumed that depth$G(M)=0$ so $b_1(x_1,M)=1$. 
		
		From the above exact sequence we have $$0\rt \mathfrak{m}^2M:x_1/\mathfrak{m}M \rt \mathfrak{m}^2M/J\mathfrak{m}M\rt \mathfrak{m}^2{M_1}/x_2\mathfrak{m}{M_1}\rt 0.$$
		So  we get $\ell(\mathfrak{m}^2M/J\mathfrak{m}M)=\rho_1+b_1(x_1,M)=2$.

		Now consider $\overline{G(M)}=G(M)/(x^*_1,x^*_2)G(M)$. So we have
		$$G(M)/(x^*_1,x^*_2)G(M)=M/\mathfrak{m}M \oplus \mathfrak{m}M/(JM+\mathfrak{m}^2M)\oplus \mathfrak{m}^2M/J\mathfrak{m}M+\mathfrak{m}^3M$$
		Since deg $h_{M_2}(z)=1$,  $\mathfrak{m}^2M\sub JM$  and  $\mathfrak{m}^3M\sub J\mathfrak{m}M$.\\
		Therefore we have
		$$\overline{G(M)}=G(M)/(x_1^*,x_2^*)G(M)=M/\mathfrak{m}M \oplus \mathfrak{m}M/JM \oplus \mathfrak{m}^2M/J\mathfrak{m}M.$$
		Its Hilbert series is $2+z+2z^2$. But this 
		is not a possible Hilbert series [from \ref{overline{G(M)}}]. So depth$G(M)\geq1$, but in this case depth$G(M)\ne 2$ because $h_{M_1}(z)\ne h_{M_2}(z)$. This gives us depth$G(M)=1$.\\ 
		Now assume dim$M\geq3$ and $\underline{x}=x_1,\ldots,x_d$ a maximal $\phi$-superficial sequence. Set $M_{d-2}=M/(x_1,\ldots,x_{d-2})M$. We now have two cases.\\ First case when $M_{d-2}$ has minimal multiplicity with $h_{M_{d-2}}=2+z$. By Sally-descent in this case we have $G(M)$ is \CM \ and $h_M(z)=2+z$.\\
		Second case when depth$G(M_{d-2})=1$ and $h_{M_{d-2}}=2+z^2$. By Sally-descent we have depth$G(M)=d-1$ and $h_M(z)=2+z^2$.

		{\bf Case(3):}$a_1=1,a_2=3.$\\
		In this case $M_d\cong Q'/(y)\oplus Q'/(y^3)$. So $h_{M_d}(z)=2+z+z^2$.\\
		We  first consider the case when dim$M=2$ because if dim$M\leq1$ there is nothing to prove.\\
		Let $\underline{x}=x_1,x_2$ be a maximal $\phi$-superficial sequence. Set $J=(x_1,x_2)$ and ${M_1}=M/x_1M$.\\
		Since dim${M_1}=1$ we can write $h$-polynomial of ${M_1}$ as $h_{{M_1}}(z)=2+(\rho_0-\rho_1)z+\rho_1z^2$, where $\rho_n = \ell(\mathfrak{m}^{n+1}{M_1}/{x_2}\mathfrak{m}^n{M_1})$. Since $red(M)\leq2$,  $\rho_n=0$ for $n\geq2$. As we know that $e({M_1})=e(M_2)=4$ and $\mu({M_1})=2$ so we have $\rho_0=\ell(\mathfrak{m}{M_1}/x_2{M_1})=2$. We also know that coefficients of $h_{{M_1}}$ are non-negative [from \ref{d=1}(1)]. So possible values of $\rho_1$ are $0,1$ and $2$.\\
		We now first show that $\rho_1=0$ or $\rho_1=2$ is not possible.\\
		Because if $\rho_1=0$ then ${M_1}$ has minimal multiplicity. This implies  $G({M_1})$ is \CM \ and $h_{M_1}(z)=2+2z$. But this   is not possible  because we have  $h_{M_1}(z)\ne h_{M_2}(z) $ (see \ref{Property}).  So $\rho_1\neq 0.$\\
		If $\rho_1=2$ then consider $\overline{G(M_1)}=G({M_1})/x_2^*G({M_1})$. 
		So we have 
		$$G({M_1})/x_2^*G({M_1})={M_1}/\mathfrak{m}{M_1}\oplus \mathfrak{m}{M_1}/(x_2{M_1}+\mathfrak{m}^2{M_1})\oplus \mathfrak{m}^2{M_1}/x_2\mathfrak{m}{M_1}$$
		Since deg$h_{M_2}(z)=2$, $\mathfrak{m}^2{M_1}\not\subseteq x_2{M_1}$. Therefore the  Hilbert series of $\overline{G(M_1)}$   is $2+z+2z^2$ and this is not a possible Hilbert series [from \ref{overline{G(M)}}]. So $\rho_1\neq2$.\\
		We now consider the $\rho_1=1$ case and 
		in this case $h_{{M_1}}(z)=2+z+z^2$. This implies $G({M_1})$ is \CM \ (see \ref{mod-sup}) and by Sally-descent $G(M)$ is \CM.
		
			Now assume dim$M\geq3$ and $\underline{x}=x_1,\ldots,x_d$ a maximal $\phi$-superficial sequence. Set $M_{d-2}=M/(x_1,\ldots,x_{d-2})M$.
			So we have $G(M_{d-2})$ is \CM.
		 By Sally-descent $G(M)$ is \CM\ and $h_M(z)=2+z+z^2$.

		{\bf Case(4):} $a_1=2,a_2=2.$\\
		In this case  $M_d\cong Q'/(y^2)\oplus Q'/(y^2)$. So $h_{M_d}(z)=2+2z$. This implies $e(M_d)=4=i(M_d)\mu(M_d)$.   Notice this equality is preserved modulo any $\phi$-superficial sequence, So for dim$M\geq1$,  $G(M)$ is \CM \ and $h_M(z)=2+2z$ (see \cite[Theorem 2]{PuMCM}). 
		
		{\bf Case(5):} $a_1=2,a_2=3.$\\
		In this case $M_d\cong Q'/(y^2)\oplus Q'/(y^3)$. So  $h_{M_d}(z)=2+2z+z^2$. \\
		We  first consider the  case when dim$M=2$ because if dim$M\leq1$ there is nothing to prove.\\
		Let $\underline{x}=x_1,x_2$ be a maximal $\phi$-superficial sequence. Set $J=(x_1,x_2)$ and ${M_1}=M/x_1M$.\\
		Since dim${M_1}=1$ we can write $h$-polynomial of ${M_1}$ as $h_{{M_1}}(z)=2+(\rho_0-\rho_1)z+\rho_1z^2$, where $\rho_n = \ell(\mathfrak{m}^{n+1}{M_1}/{x_2}\mathfrak{m}^n{M_1})$. Since $red(M)\leq2$,  $\rho_n=0$ for $n\geq2$. As we know that $e({M_1})=e(M_2)=5$ and $\mu({M_1})=2$ so we have $\rho_0=\ell(\mathfrak{m}{M_1}/x_2{M_1})=3$. We also know that coefficients of $h_{{M_1}}$ are non-negative [from \ref{d=1}(1)]. \\ 
		Since  $i(M)=i({M_1})=i(M_2)=2$ (see \ref{i(M)}).  This implies $im(\phi\otimes Q/(x_1))\sub\mathfrak{m}^2(Q/(x_1)\oplus Q/(x_1))$. So we get  $\ell(\mathfrak{m}{M_1}/{\mathfrak{m}^2{M_1}})=4$, because $M_1=$coker$(\phi\otimes Q/(x_1))$. From the Hilbert series of $M_1$ we get $\ell(\mathfrak{m}{M_1}/{\mathfrak{m}^2{M_1}})=2+(\rho_0-\rho_1)=4$. This implies that  $\rho_1=1$ and $h_{{M_1}}(z)=2+2z+z^2$. So $G({M_1})$ is \CM. Now  by Sally-descent $G(M)$ is \CM. 
		
			Now assume dim$M\geq3$ and $\underline{x}=x_1,\ldots,x_d$ a maximal $\phi$-superficial sequence. Set $M_{d-2}=M/(x_1,\ldots,x_{d-2})M$. So we have $G(M_{d-2})$ is \CM.
			By Sally-descent $G(M)$ is \CM \ and $h_M(z)=2+2z+z^2$.
		
		{\bf Case(6):} $a_1=3,a_2=3$\\
		In this case $M_d\cong Q'/(y^3)\oplus Q'/(y^3)$, so  $h_{M_d}(z)=2+2z+2z^2$. This  implies $e(M_d)=i(M_d)\mu(M_d)=6$. For dim$M\geq1$ we have $e(M)=\mu(M)i(M)$,  because this equality is preserved modulo any $\phi$-superficial sequence. So  we have $G(M)$ is \CM \ and $h_M(z)=2+2z+2z^2$ (see \cite[Theorem 2]{PuMCM}). 
	\end{proof}
	So from above theorem it is clear that:
	\begin{enumerate}
		\item If $a_1=1 ,a_2=1 $ then $G(M)$ is \CM\ and $h_M(z)=2$.

		\item If $a_1=1 ,a_2=2 $ then we have two case:
		\begin{enumerate}
			\item $G(M)$ is \CM \ if and only if $h_M(z)=2+z$
			\item  depth$G(M)=d-1$ if and only if $h_M(z)=2+z^2$
		\end{enumerate}

		\item If $a_1=1 ,a_2= 3$ then
		$G(M)$ is \CM\  and $h_M(z)=2+z+z^2$.

		\item If $a_1=2 ,a_2=2 $ then $G(M)$ is \CM \ and $h_M(z)=2+2z.$

		\item If $a_1=2 ,a_2=3 $ then
		$G(M)$ is \CM \ and $h_M(z)=2+2z+z^2$

		\item If $a_1=3 ,a_2= 3$ then $G(M)$ is \CM \ and $h_M(z)=2+2z+2z^2.$
	\end{enumerate}
	
	{\bf Note:} red$(M)<2$ occurs only in cases (1), (2a) and (4).
	
	We can conclude that:
	\begin{corollary}
		Let $(A,\mathfrak{m})$ be a hypersurface ring  of dimension $d\geq 2$ with infinite residue field and  $M$ an MCM $A$-module with $\mu(M)=2$. Now if
		$red(M)=2$ then depth$G(M)\geq d-1$. Also, possible $h$-polynomials are
		\[ h_M(z) = \begin{cases} 
		2+z^2 & depth G(M) = d-1. \\
		2+z+z^2 & G(M)\ \text{is Cohen-Macaulay}. \\
		2+2z+z^2 &  G(M)\ \text{is Cohen-Macaulay}. \\
		2+2z+2z^2 &  G(M)\ \text{is Cohen-Macaulay}.
		
		\end{cases}
		\]
		
	\end{corollary}

\section{\bf The case when $\mu(M)=3$ }

\begin{theorem}\label{muM=3}
	Let $(A,\mathfrak{m})$ be a hypersurface ring  of dimension $d$ with infinite residue field and  $M$ an MCM $A$-module with $\mu(M)=3$.
	Now if  $red(M)\leq2$, then depth$G(M)\geq d-2$.
\end{theorem}

\begin{proof} By \ref{Base change} we can assume that $A$ is a complete local ring.\\
	So, $(A,\mathfrak{m})=(Q/(f),\mathfrak{n}/(f))$, where
	$(Q,\mathfrak{n})$ is a regular local ring of dimension $d+1$ and $f\in \mathfrak{n}^i\setminus\mathfrak{n}^{i+1}$ for some $i\geq2$.\\
	Let dim $M\geq1$ and $0\rt Q^3\xrightarrow{\phi} Q^3\rt M\rt 0$ be a minimal presentation of $M$. Let $\underline{x}=x_1,\ldots,x_d$ be a maximal $\phi$-superficial sequence (see \ref{phi}). Set $M_d=M/\underline{x}M$ and $(Q',(y))=(Q/(\underline{x}),\mathfrak{n}/(\underline{x}))$.
	
	Clearly, $Q'$ is DVR and so $M_d\cong Q'/(y^{a_1})\oplus Q'/(y^{a_2})\oplus Q'/(y^{a_3})$. As $red(M)\leq2$ we can assume that $1\leq a_1\leq a_2\leq a_3\leq 3$. Now we  consider all possibilities separately

	{\bf Case(1):} $a_1=1,a_2=1, a_3=1.$\\
	In this case $M_d\cong Q'/(y)\oplus Q'/(y)\oplus Q'/(y)$. This gives $h_{M_d}(z)=3$. So
	$e(M_d)=\mu(M_d)=3$.  For  dim$M\geq1$ we know that $e(M)=e(M_d)$ and $\mu(M)=\mu(M_d)$. Also, notice that $i(M)=i(M_d)=1$. So we have $e(M)=\mu(M)=3$ this implies that $M$ is an  Ulrich module.  So $G(M)$ is \CM \ and $h_M(z)=3$.  (see \cite[Theorem 2]{PuMCM}).
	
	{\bf Case(2):} $a_1=1,a_2=1, a_3=2.$\\
	In this case $M_d\cong Q'/(y)\oplus Q'/(y)\oplus Q'/(y^2)$. So we have $h_{M_d}(z)=3+z$.\\
	We  first consider the case when dim$M=3$ because if dim$M\leq2$ there is nothing to prove.\\
	Let $\underline{x}=x_1,x_2,x_3$ be a maximal $\phi$-superficial sequence. Set $M_1=M/x_1M$, $M_2=M/(x_1,x_2)M$ and $J=(x_2,x_3)$.
	
	Since dim$M_2=1$ then we can write $h$-polynomial of $M_2$ as $h_{M_2}(z)=3+(\rho_0-\rho_1)z+\rho_1z^2$ where $\rho_n=\ell(\mathfrak{m}^{n+1}M_2/{x_3\mathfrak{m}^nM_2})$.  Since red$(M)\leq 2$, we have $\rho_n=0$ for all $n\geq 2$. We have $\rho_0=\ell(\mathfrak{m}M_2/x_3M_2)=1$ and we know that coefficients of $h_{M_2}$ are non-negative [from \ref{d=1}(1)]. So possible values of  $\rho_1$ are $0$ and $1$.\\
	{\bf Subcase(i): $\rho_1=0$.}\\
	Notice that in this case $M_2$ has minimal multiplicity and this implies $G(M_2)$ is \CM. Now by Sally-descent $G(M)$ is Cohen-Macaulay and $h_M(z)=3+z$. \\
	{\bf Subcase(ii): $\rho_1=1$.} \\ In this case $h_{M_2}(z)=3+z^2$. This implies depth$G(M_2)=0$ because $h_{M_2}(z)\ne h_{M_3}(z) $ (see \ref{Property}). \\	
	Now since dim$M_1=2$, from \ref{mod-sup}  we have $$h_{M_1}(z)=h_{{M_2}}(z)-(1-z)^2b_{x_2,M_1}(z)$$
	This gives us $$e_2(M_1)=e_2({M_2})-\sum b_i(x_2,M_1)$$
	where $b_i(x_2,M_1)=\ell(\mathfrak{m}^{i+1}M_1:x_2/\mathfrak{m}^iM_1)$.\\
	We know that $e_2(M_1)$  and $\sum b_i(x_2,M_1) $ are non-negative integers. Also in this case $e_2({M_2})=1$. This implies $\sum b_i(x_2,M_1) \leq 1 $.  \\
	Now from exact sequence (see \ref{exact seq}) 
	\begin{align*}
	0 \rt \mathfrak{m}^nM_1:J/\mathfrak{m}^{n-1}M_1\rt \mathfrak{m}^nM_1:x_2/\mathfrak{m}^{n-1}M_1 & \rt \mathfrak{m}^{n+1}M_1:x_2/\mathfrak{m}^{n}M_1\\
	\rt \mathfrak{m}^{n+1}M_1/J\mathfrak{m}^nM_1
	&\rt \mathfrak{m}^{n+1}{M_2}/x_3\mathfrak{m}^n{M_2}\rt 0
	\end{align*}
	we have, if $b_1(x_2,M_1)=0$ then $b_i(x_2,M_1)=0$ for all $i\geq2$.\\
	{\bf Claim:} depth$G(M_1)=1$.\\
	{\bf Proof of the claim:}	Now if possible assume that depth$G(M_1)=0$ then from \ref{Property},  $\sum b_i(x_2,M_1) =1$.
	So in this case we have $b_1(x_2,M_1)=1$.\\
	From the above exact sequence we have  
	$$0\rt \mathfrak{m}^2M_1:x_2/\mathfrak{m}M_1\rt\mathfrak{m}^2M_1/J\mathfrak{m}M_1\rt\mathfrak{m}^2{M_2}/x_3\mathfrak{m}{M_2}\rt0$$
	So, $\lambda(\mathfrak{m}^2M_1/J\mathfrak{m}M_1)=\rho_1+b_1(x_2,M_1)=2$.

	Now consider $\overline{G(M_1)}=G(M_1)/(x_2^*,x_3^*)G(M_1)$. So we have
	$$G(M_1)/(x_2^*,x_3^*)G(M_1)=M_1/\mathfrak{m}M_1 \oplus \mathfrak{m}M_1/(JM_1+\mathfrak{m}^2M_1)\oplus \mathfrak{m}^2M_1/(J\mathfrak{m}M_1+\mathfrak{m}^3M_1)$$
	Now since deg$h_{M_3}(z)=1$,  $\mathfrak{m}^2M_1\subseteq JM_1$ and $\mathfrak{m}^3M_1=J\mathfrak{m}^2M_1\sub J\mathfrak{m}M_1$ . \\
	Thus we have
	$$\overline{G(M_1)}=G(M_1)/(x_2^*,x_3^*)G(M_1)=M_1/\mathfrak{m}M_1 \oplus \mathfrak{m}M_1/JM_1 \oplus \mathfrak{m}^2M_1/J\mathfrak{m}M_1$$
	Now since  $\ell(\mathfrak{m}M_1/JM_1)=1$,  the Hilbert series of $\overline{G(M_1)}$  is $3+z+2z^2$. But  this is not a possible Hilbert series [from \ref{overline{G(M)}}]. This implies depth$G(M_1)\geq 1$.
	
	So we can conclude from here that when $\rho_1=1$ then depth$G(M_1)=1$. Notice that depth$G(M_1)\ne 2$ because depth$G(M_2)=0$. By Sally-descent in this case we have  depth$G(M)= 2$ and $h_M(z)=3+z^2$.
	
	Now assume  dim$M\geq4$. Let $\underline{x}=x_1\ldots,x_d$ be a maximal $\phi$-superficial sequence. Set $M_{d-3}=M/(x_1,\ldots,x_{d-3})M$. We now have two cases.\\
	First case when $G(M_{d-3})$ is \CM\ and $h_{M_{d-3}}(z)=3+z$. By Sally-descent we have $G(M)$ is \CM\ and $h_{M}(z)=3+z$.\\
	Second case when depth$G(M_{d-3})=2$ and $h_{M_{d-3}}(z)=3+z^2$. By Sally-descent we have depth$G(M)=d-1$ and  $h_{M}(z)=3+z^2$.
	
{\bf Case(3):} $a_1=1,a_2=1, a_3=3.$\\
	In this case $M_d\cong Q'/(y) \oplus Q'/(y) \oplus Q'/(y^3)$ this implies that $h_{M_d}(z)=3+z+z^2$.\\
	We  first consider the case when dim$M=3$  because if dim$M\leq2$ there is nothing to prove.\\
	Let $\underline{x}=x_1,x_2,x_3$ be a $\phi$-superficial sequence.\\ Set $M_1=M/x_1M$, $M_2=M/(x_1,x_2)M$ and $J=(x_2,x_3)$.\\
	Since dim$M_2=1$ we can write $h$-polynomial of $M_2$ as $h_{M_2}(z)=3+(\rho_0-\rho_1)z+\rho_1z^2$ where $\rho_n=\ell(\mathfrak{m}^{n+1}M_2/{x_3\mathfrak{m}^nM_2})$. Now since red$(M)\leq2$, $\rho_n=0$ for all $n\geq 2$. We know that $e(M_2)=e(M_3)=5$ therefore $\rho_0=\ell(\mathfrak{m}M_2/x_3M_2)=2$ and coefficients of $h_{M_2}$ are non-negative [from \ref{d=1}(1)]. So possible values of $\rho_1$ are $0,1$ and $2$.\\
	We first show that $\rho_1=0$ or $\rho_1=2$ is not possible.
	
	If $\rho_1=0$ then $M_2$ has minimal multiplicity. This implies $G(M_2)$ is \CM. But  this  is a contradiction, because $h_{M_3}(z)\neq h_{M_2}$ (see \ref{Property}). So $\rho_1\ne 0$.\\
	If $\rho_1=2$ then consider $G(M_2)/x_3^*G(M_2)$.  So we have
	$$G(M_2)/x_3^*G(M_2)=M_2/\mathfrak{m}M_2\oplus \mathfrak{m}M_2/(x_3M_2+\mathfrak{m}^2M_2)\oplus \mathfrak{m}^2M_2/x_3\mathfrak{m}M_2$$
	Now since $\mathfrak{m}^2M_2\not\subseteq x_3M_2$,  its Hilbert series is $3+z+2z^2$. But  this is not a possible Hilbert series [from \ref{overline{G(M)}}]. So in this case we get a contradiction. This implies $\rho_1\ne 2$.
	
	Now we consider the case when  $\rho_1=1$. In this case 
	$h_{M_2}(z)=3+z+z^2$. This implies $G(M_2)$ is \CM, because $h_{M_2}(z)=h_{M_3}(z)$ (see \ref{mod-sup}). Now by Sally-descent $G(M)$ is \CM. 
	
	Now assume  dim$M\geq4$. Let $\underline{x}=x_1\ldots,x_d$ be a maximal $\phi$-superficial sequence. Set $M_{d-3}=M/(x_1,\ldots,x_{d-3})M$. So we have 
	$G(M_{d-3})$ is \CM\ and $h_{M_{d-3}}(z)=3+z+z^2$. By Sally-descent we have $G(M)$ is \CM\ and $h_{M}(z)=3+z+z^2$.

	{\bf Case(4):} $a_1=1,a_2=2, a_3=2.$\\
	In this case  $M_d\cong Q'/(y)\oplus Q'/(y^2) \oplus Q'/(y^2)$. So $h_{M_d}(z)=3+2z.$\\
	We  first consider the case when dim$M=3$ because if dim$M\leq2$ there is nothing to prove.\\
	Let $\underline{x}=x_1,x_2,x_3$ be a maximal $\phi$-superficial sequence. Set $M_1=M/x_1M$, $M_2=M/(x_1,x_2)M$ and $J=(x_2,x_3)$ .\\
	We  first prove two claims:\\
	{\bf Claim(1):} $\widetilde{\mathfrak{m}^iM_2}=\mathfrak{m}^iM_2$ for all $i\geq2$.\\
	{\bf Proof} Since $\mathfrak{m}^{i+1}M_2=x_3\mathfrak{m}^iM_2$ for all $i\geq2$ so $\mathfrak{m}^{i+1}M_2:x_3=\mathfrak{m}^iM_2$ for all $i\geq2$.
	
	So from \ref{RR-2} for all $i\geq2$  we have
	$$0\rt \widetilde{\mathfrak{m}^iM_2}/\mathfrak{m}^iM_2\rt \widetilde{\mathfrak{m}^{i+1}M_2}/\mathfrak{m}^{i+1}M_2.$$ We also know that for $i\ggg 0$, $\widetilde{\mathfrak{m}^iM_2}=\mathfrak{m}^iM_2$. So it is clear that $\widetilde{\mathfrak{m}^iM_2}=\mathfrak{m}^iM_2$ for all $i\geq2$.\\
	{\bf Claim(2):} $\ell(\widetilde{\mathfrak{m}M_2}/\mathfrak{m}M_2)\leq1$.\\
	{\bf Proof:} Now since $\mu(M_2)=3$ so $\ell(\widetilde{\mathfrak{m}M_2}/\mathfrak{m}M_2)\leq3$. If $\ell(\widetilde{\mathfrak{m}M_2}/\mathfrak{m}M_2)=3$ then $M_2=\widetilde{\mathfrak{m}M_2}$ this implies that $\mathfrak{m}M_2=\mathfrak{m}\widetilde{\mathfrak{m}M_2}\sub\widetilde{\mathfrak{m}^2M_2}=\mathfrak{m}^2M_2$. So $\mathfrak{m}M_2=0$ and this is a contradiction. Now if possible assume that $\ell(\widetilde{\mathfrak{m}M_2}/\mathfrak{m}M_2)=2$  then $M_2=\langle m,l_1,l_2\rangle $ with $l_1,l_2\in \widetilde{\mathfrak{m}M_2}\setminus\mathfrak{m}M_2$.  Now we have  $l_i\mathfrak{m}\sub \widetilde{\mathfrak{m}^2M_2}=\mathfrak{m}^2M_2 $ for $i=1,2.$ If we set $\mathfrak{m}'=\mathfrak{m}/(x_1,x_2,x_3)$, then $\mathfrak{m}'$ is principal ideal. We also know that $\ell(\mathfrak{m}M_3)=\ell(\mathfrak{m}'M_3)$ and $\mathfrak{m}^2M_3=(\mathfrak{m}')^2M_3=0 $, this implies that $\ell(\mathfrak{m}M_3)=1$. But we have $\ell(\mathfrak{m}M_3)=2$ (from the Hilbert series of $M_3$). So $\ell(\widetilde{\mathfrak{m}M_2}/\mathfrak{m}M_2)\leq1$. 
	
	Now from \ref{RR-2} we have 
	$$0\rt (\mathfrak{m}^2M_2:x_3)/\mathfrak{m}M_2\rt \widetilde{\mathfrak{m}M_2}/\mathfrak{m}M_2 .$$
	So from claim(2), we have $b_1(x_3,M_2)=\ell(\mathfrak{m}^2M_2:x_3/\mathfrak{m}M_2)\leq 1$.\\
	Since dim$M_2=1$ we can write $h$-polynomial of $M_2$ as $h_{M_2}(z)=3+(\rho_0-\rho_1)z+\rho_1z^2$ where $\rho_n=\ell(\mathfrak{m}^{n+1}M_2/{x_3\mathfrak{m}^nM_2})$. We have  $\rho_0=\ell(\mathfrak{m}M_2/x_3M_2)=2$ and coefficients of $h_{M_2}$ are non-negative [from \ref{d=1}(1)].
	
	From short exact sequence (see \ref{exact d})
	$$0\rt \mathfrak{m}^2M_2:x_3/\mathfrak{m}M_2\rt \mathfrak{m}^2M_2/x_3\mathfrak{m}M_2\rt \mathfrak{m}^2M_3/0\rt 0$$
	we have $\rho_1=b_1(x_3,M_2)$ because $\mathfrak{m}^2M_3=0$.\\ From Claim(2) we have $b_1(x_3,M_2)\leq1$.\\ 
	Now we have two cases.\\
	{\bf Subcase (i):} When  $b_1(x_3,M_2)=0$.\\
	This implies 
	$\rho_1=0$ and so in this case $M_2$ has minimal multiplicity. Therefore $G(M_2)$ is \CM.  By Sally-descent $G(M)$ is \CM \ and $h_M(z)=3+2z$.\\   
	{\bf Subcase (ii):} When $b_1(x_3,M_2)\neq0$.\\ From Claim(2) we have $\rho_1=b_1(x_3,M_2)=1$. So in this case $h_{M_2}(z)=3+z+z^2$. This implies $h_{M_2}(z)\ne h_{M_3}(z)$. So depth$G(M_2)=0$ (see \ref{Property}). 
	
	Since dim$M_1=2$, from \ref{mod-sup}  we have $$h_{M_1}(z)=h_{{M_2}}(z)-(1-z)^2b_{x_2,M_1}(z)$$
	This gives us $$e_2(M_1)=e_2({M_2})-\sum b_i(x_2,M_1),$$
	where $\sum b_i(x_2,M_1)=\ell(\mathfrak{m}^{i+1}M_1:x_2/\mathfrak{m}^iM)$.\\
	We know that $e_2(M_1)$  and $\sum b_i(x_2,M_1) $ are non-negative integers. Also in this case $e_2({M_2})=1.$ So we have $\sum b_i(x_2,M_1)\leq 1$. 
	
	Since red$(M)\leq2$, from exact sequence (see \ref{exact seq}) 
	\begin{align*}
	0 \rt \mathfrak{m}^nM_1:J/\mathfrak{m}^{n-1}M_1\rt \mathfrak{m}^nM_1:x_2/\mathfrak{m}^{n-1}M_1 & \rt \mathfrak{m}^{n+1}M_1:x_2/\mathfrak{m}^{n}M_1\\
	\rt \mathfrak{m}^{n+1}M_1/J\mathfrak{m}^nM_1
	&\rt \mathfrak{m}^{n+1}{M_2}/x_3\mathfrak{m}^n{M_2}\rt 0
	\end{align*}
	we get, if $b_1(x_2,M_1)=0$ then  $b_i(x_2,M_1)=0$ for all $i\geq2$.
	
	{\bf Subcase (ii).(a):} When $\sum b_i(x_2,M_1) =0$.\\
	Now from \ref{Property},  depth$G(M_1)\geq1$. Notice that $G(M_1)$ cannot be a \CM \ module, because depth$G(M_2)=0$. So in this case depth$G(M_1)=1$. By Sally-descent depth$G(M)=2 $  
	and $h_M(z)=3+z+z^2$. 
	
	{\bf Subcase (ii).(b):} When $\sum b_i(x_2,M_1) \ne 0$.\\
	This implies $b_1(x_2,M_1)=\ell((\mathfrak{m}^2M_1:x_2)/\mathfrak{m}M_1)=1$.
	So in this case we have depth$G(M_1)=0$ (see \ref{Property}).

	From the above exact sequence we get 
	$$0\rt \mathfrak{m}^2M_1:x_2/\mathfrak{m}M_1\rt \mathfrak{m}^2M_1/J\mathfrak{m}M_1\rt \mathfrak{m}^2M_2/x_3\mathfrak{m}M_2\rt 0.$$
	
	So we have $\ell(\mathfrak{m}^2M_1/J\mathfrak{m}M_1)=\rho_1+\ell((\mathfrak{m}^2M_1:x_2)/\mathfrak{m}M_1)=2$. We can write $h$-polynomial of $M_1$ as  $h_{M_1}(z)=h_{M_2}(z)-(1-z)^2z=3+3z^2-z^3$ (see \ref{mod-sup}). 
	
	Consider $\overline{G(M)}=G(M)/(x_1^*,x_2^*,x_3^*)G(M)$. So we have
	$$\overline{G(M)}=G(M)/(x_1^*,x_2^*,x_3^*)G(M)=M/\mathfrak{m}M \oplus \mathfrak{m}M/(\underline{x})M \oplus \mathfrak{m}^2M/(\underline{x})\mathfrak{m}M$$
	Since $\ell(\mathfrak{m}M/(\underline{x})M)=2$ by looking at the Hilbert series of $\overline{G(M)}$ we can say that $\ell(\mathfrak{m}^2M/(\underline{x})\mathfrak{m}M)\leq2$ [from \ref{overline{G(M)}}].
	
	From exact sequence (see \ref{exact d}) $$0\rt (\mathfrak{m}^2M:x_1)/\mathfrak{m}M\rt \mathfrak{m}^2M/(\underline{x})\mathfrak{m}M\rt \mathfrak{m}^2M_1/J\mathfrak{m}M_1\rt 0$$
	we have $\ell(\mathfrak{m}^2M/(\underline{x})\mathfrak{m}M)=2$ and $\mathfrak{m}^2M:x_1=\mathfrak{m}M$ because $\ell(\mathfrak{m}^2M_1/J\mathfrak{m}M_1)=2$.
	
	Now consider $$\delta=\sum\ell(\mathfrak{m}^{n+1}M\cap (\underline{x})M/(\underline{x})\mathfrak{m}^nM).$$
	
	We know that if $\delta\leq2$ then depth$G(M)\geq d-\delta$ (see \cite[Theorem 5.1]{apprx}). In our case $\delta=2$, because $\mathfrak{m}^2M\sub (\underline{x})M$ and $\ell(\mathfrak{m}^2M/(\underline{x})\mathfrak{m}M)=2$. So depth$G(M)\geq 1$. Notice that in this case depth$G(M)=1$ because depth$G(M_1)=0$.\\ Now assume dim$M\geq4$ and $\underline{x}=x_1,\ldots,x_d$ a maximal $\phi$-superficial sequence. Set $M_{d-3}=M/(x_1,\ldots,x_{d-3})M$. We now have three cases.\\
	First case when $G(M_{d-3})$ is \CM\ and $h_{M_{d-3}}=3+2z$. By Sally-descent $G(M)$ is \CM\ and $h_M(z)=3+2z$.\\
	Second case when depth$G(M_{d-3})=2$. By Sally-descent depth$G(M)=d-1$ and $h_M(z)=3+z+z^2$.\\
	Third case when depth$G(M_{d-3})=1$. By Sally-descent depth$G(M)=d-2$ and $h_M(z)=3+3z^2-z^3$.

	{\bf Case(5):} $a_1=1,a_2=2, a_3=3.$\\
	In this case $M_d\cong Q'/(y) \oplus Q'/(y^2) \oplus Q'/(y^3)$ this implies that $h_{M_d}(z)=3+2z+z^2$.\\
	We  first consider the case when  dim$M=3$ because if dim$M\leq2$ there is nothing to prove.\\
	Let $\underline{x}=x_1,x_2,x_3$ be a maximal $\phi$-superficial sequence. Set $M_1=M/x_1M$, $M_2=M/(x_1,x_2)M$ and $J=(x_2,x_3)$.\\ 
	Since dim$M_2=1$ we can write $h$-polynomial of $M_2$ as $h_{M_2}(z)=3+(\rho_0-\rho_1)z+\rho_1z^2$ where $\rho_n=\ell(\mathfrak{m}^{n+1}M_2/{x_3\mathfrak{m}^nM_2})$. Since red$(M)\leq 2$, we get $\rho_n=0$ for all $n\geq 2$.\\ Since  $\rho_0=\ell(\mathfrak{m}M_2/x_3M_2)=3$ and coefficients of $h_{M_2}$ are non-negative[from \ref{d=1}(1)], so possible values of $\rho_1$ are $0,1,2$ and $3$.
	
	We first show that $\rho_1=0$ or $\rho_1=3$ is not possible.
	
	If $\rho_1=0$ then  $M_2$ has minimal multiplicity so $G(M_2)$ is \CM\ which is a contradiction because in this case  $h_{M_3}(z)\ne h_{M_2}(z)$ (see \ref{Property}). So $\rho_1\ne 0$.\\
	If $\rho_1=3$  consider $\overline{G(M_2)}=G(M_2)/x_3^*G(M_2)$.
	So we have
	$$G(M_2)/x_3^*G(M_2)=M_2/\mathfrak{m}M_2\oplus \mathfrak{m}M_2/x_3M_2+\mathfrak{m}^2M_2\oplus \mathfrak{m}^2M_2/x_3\mathfrak{m}M_2.$$
	Since deg$h_{M_3}(z)=2$, $\mathfrak{m}^2M_2\not\subseteq x_3M_2$. So its Hilbert series is $3+\alpha z+3z^2$ where $\alpha \leq 2$, but this is not a possible Hilbert series [from \ref{overline{G(M)}}]. So $\rho_1\neq 3$.\\
	So  possible values of $\rho_1$ are $1$ and $2$.\\
	{\bf Subcase(i): $\rho_1=1$.} \\
	In this case $h_{M_2}(z)=3+2z+z^2$. So $h_{M_2}(z)=h_{M_3}(z)$ and this implies $G(M_2)$ is \CM\ (see \ref{Property}). By Sally-descent $G(M)$ is \CM.\\
	{\bf Subcase(ii): $\rho_1=2$.}\\
	In this case $h_{M_2}(z)=3+z+2z^2$ and depth$G(M_2)=0$ because $h_{M_2}(z)\ne h_{M_3}(z)$ (see \ref{Property}).
	
	Since dim$M_1=2$,  we have (see \ref{Property})  $$e_2(M_1)=e_2({M_2})-\sum b_i(x_2,M_1),$$
	where $b_i(x_2,M_1)=\ell(\mathfrak{m}^{i+1}M_1:x_2/\mathfrak{m}^iM) $. 
	We know that $e_2(M_1)$ and $\sum b_i(x_2,M_1) $ are non-negative integers and   $e_2({M_2})=2$. This implies $\sum b_i(x_2,M_1) \leq 2$.
	
	Since red$(M)\leq2$, from exact sequence (see \ref{exact seq}) 
	\begin{align*}
	0 \rt \mathfrak{m}^nM_1:J/\mathfrak{m}^{n-1}M_1\rt \mathfrak{m}^nM_1:x_2/\mathfrak{m}^{n-1}M_1 & \rt \mathfrak{m}^{n+1}M_1:x_2/\mathfrak{m}^{n}M_1\\
	\rt \mathfrak{m}^{n+1}M_1/J\mathfrak{m}^nM_1
	&\rt \mathfrak{m}^{n+1}{M_2}/x_3\mathfrak{m}^n{M_2}\rt 0
	\end{align*}
	we get, if $b_1(x_2,M_1)=0$ then  $b_i(x_2,M_1)=0$ for all $i\geq2$.\\
	{\bf Claim:} depth$G(M_1)=1$.\\
	{\bf Proof of the claim:}
	If possible assume that depth$G(M_1)=0$. This implies   $ b_1(x_2,M_1) \geq 1$ (see \ref{Property}).

	Consider $\overline{G(M_1)}=G(M_1)/(x_2^*,x_3^*)G(M_1)$, so we have
	$$G(M_1)/(x_2^*,x_3^*)G(M_1)=M_1/\mathfrak{m}M_1 \oplus \mathfrak{m}M_1/(JM_1+\mathfrak{m}^2M_1)\oplus \mathfrak{m}^2M_1/(J\mathfrak{m}M_1+\mathfrak{m}^3M_1)$$
	Therefore we have 	$$\overline{G(M_1)}=M_1/\mathfrak{m}M_1 \oplus \mathfrak{m}M_1/(JM_1+\mathfrak{m}^2M_1) \oplus \mathfrak{m}^2M_1/J\mathfrak{m}M_1.$$
	From the above  exact sequence we get
	$$0\rt \mathfrak{m}^2M_1:x_2/\mathfrak{m}M_1\rt \mathfrak{m}^2M_1/J\mathfrak{m}M_1\rt \mathfrak{m}^2M_2/x_3\mathfrak{m}M_2\rt 0.$$
	So we have $\ell(\mathfrak{m}^2M_1/J\mathfrak{m}M_1)=\rho_1+b_1(x_2,M_1)$.

	So Hilbert series of $\overline{G(M_1)}$ is $3+\alpha z+(\rho_1+b_1(x_2,M_1))z^2$ where   $\rho_1+b_1(x_2,M_1)\geq3$ and $\alpha\leq 2$,  because we know that
	$\ell({\mathfrak{m}M_1}/{JM_1})=3$ and $\mathfrak{m}^2M_1\not\sub JM_1$. So this is not a possible Hilbert series [from \ref{overline{G(M)}}]. So  this implies depth$G(M_1)\geq 1$. Notice that $G(M_1)$ cannot be a \CM\ module, because depth$G(M_2)=0$. So depth$G(M_1)=1$ By Sally-descent depth$G(M)=2$.
	
	Now assume dim$M\geq4$ and $\underline{x}=x_1,\ldots,x_d$ a maximal $\phi$-superficial sequence. Set $M_{d-3}=M/(x_1,\ldots,x_{d-3})M$. We now have two cases.\\
	First case when $G(M_{d-3})$ is \CM. By Sally-descent $G(M)$ is \CM\ and $h_M(z)=3+2z+z^2$.\\
	Second case when depth$G(M_{d-3})=2$. By Sally-descent depth$G(M)=d-1$ and $h_M(z)=3+z+2z^2$.
	
	{\bf Case(6):} $a_1=a_2=a_3=2$.\\
	In this case we have  $M_d\cong Q'/(y^2)\oplus Q'/(y^2)\oplus Q'/(y^2)$ and $h_{M_d}(z)=3+3z$. So we have $e(M_d)=i(M_d)\mu(M_d)=6$. For dim$M\geq1$ we know  $e(M)=i(M)\mu(M)=6$, because it is preserved modulo any $\phi$-superficial sequence.   This implies $G(M)$  is \CM\ and $h_M(z)=3+3z$  (see \cite[Theorem 2]{PuMCM}).
	
	{\bf Case(7):} $a_1=1,a_2=3,a_3=3.$\\
	In this case  $M_d\cong Q'/(y)\oplus Q'/(y^3) \oplus Q'/(y^3)$ and $h_{M_d}(z)=3+2z+2z^2$.\\
	We  first consider the case when dim$M=3$ because if dim$M\leq2$ there is nothing to prove.\\
	Let $\underline{x}=x_1,x_2,x_3$ be a maximal $\phi$-superficial sequence. Set $M_1=M/x_1M$, $M_2=M/(x_1,x_2)M$ and $J=(x_2,x_3)$.\\ 
	Since dim$M_2=1$ we can write $h$-polynomial of $M_2$ as $h_{M_2}(z)=3+(\rho_0-\rho_1)z+\rho_1z^2$ where $\rho_n=\ell(\mathfrak{m}^{n+1}M_2/{x_3\mathfrak{m}^nM_2})$. Since  $\rho_0=\ell(\mathfrak{m}M_2/x_3M_2)=4$ and coefficients of $h_{M_2}$ are non-negative [from \ref{d=1}(1)]. So possible values of $\rho_1$ are $0,1,2,3$ and $4$.
	
	We first show that $\rho_1=0$, $\rho_1=1$ or $\rho_1=4$ is not possible.\\
	If $\rho_1=0$ then $M_2$ has minimal multiplicity with $h_{M_2}(z)=3+4z$. So $G(M_2)$ is \CM. But this  is a contradiction because $h_{M_3}(z)\ne h_{M_2}(z) $ (see \ref{Property}). So $\rho_1\ne0$.\\
	If $\rho_1=1$ then $h_{M_2}(z)=3+3z+z^2$. This implies  depth$G(M_2)=0$, because $h_{M_2}(z)\ne h_{M_3}(z)$ (see \ref{Property}). So we have 
	$$\ell({\mathfrak{m}^2M_2}/{x_3\mathfrak{m}^2M_2})=\ell({\mathfrak{m}^2M_2}/{x_3\mathfrak{m}M_2})+\ell(\mathfrak{m}M_2/\mathfrak{m}^2M_2)=1+6=7$$
	and $\ell({\mathfrak{m}^2M_3})=2$.
	
	From   short exact sequence (see \ref{exact1})$$0\rt {\mathfrak{m}^2M_2: x_3}/{\mathfrak{m}^2M_2}\rt {\mathfrak{m}^2M_2}/{x_3\mathfrak{m}^2M_2}\rt {{\mathfrak{m}^2M_3}}/{0}\rt 0$$ we get $\ell({\mathfrak{m}^2M_2: x_3}/{\mathfrak{m}^2M_2})=5$ but this  is a contradiction because $ \ell(\mathfrak{m}^2M_2:x_3/\mathfrak{m}^2M_2)\geq \ell(\mathfrak{m}M_2/\mathfrak{m}^2M_2)$ and $\ell(\mathfrak{m}M_2/\mathfrak{m}^2M_2)=6$. So $\rho_1\ne1$. \\
	If $\rho_1=4$, consider $$G(M_2)/x_3^*G(M_2)=M_2/\mathfrak{m}M_2\oplus \mathfrak{m}M_2/(x_3M_2+\mathfrak{m}^2M_2)\oplus \mathfrak{m}^2M_2/x_3\mathfrak{m}M_2$$
	Its Hilbert series is $3+\alpha z+4z^2$ where $\alpha\leq3$, because $\mathfrak{m}^2M_2\not\subseteq x_3M_2$ and $\rho_0=\ell(\mathfrak{m}M_2/x_3M_2)=4$.  So this is not a possible Hilbert series [from \ref{overline{G(M)}}]. Therefore, $\rho_1\neq4$.
	
	Now the possible values of $\rho_1$ are $2$ and $3$.
	
	{\bf Subcase(i): $\rho_1=2$. }\\
	In this case we have  $h_{M_2}(z)=3+2z+2z^2$ and so $G(M_2)$ is \CM. By Sally-descent $G(M)$ is \CM.
	
	{\bf Subcase(ii): $\rho_1=3$. }\\
	In this case we have $h_{M_2}(z)=3+z+3z^2$. So  depth$G(M_2)=0$, because  $h_{M_2}(z)\ne h_{M_3}(z)$ (see \ref{Property}).\\
	Since  dim$M_1=2$, we have ( from \ref{Property}) $$e_2(M_1)=e_2({M_2})-\sum b_i(x_2,M_1),$$
	where $b_i(x_2,M_1)=\ell(\mathfrak{m}^{i+1}M_1:x_2/\mathfrak{m}^iM_1)$.
	We know that $e_2(M_1)$ and $\sum b_i(x_2,M_1) $ are non-negative integers and in this case $e_2({M_2})=3$. This implies $\sum b_i(x_2,M_1) \leq3$.
	
	Since red$(M)\leq2$, from exact sequence (\ref{exact seq})
	\begin{align*}
	0 \rt \mathfrak{m}^nM_1:J/\mathfrak{m}^{n-1}M_1\rt \mathfrak{m}^nM_1:x_2/\mathfrak{m}^{n-1}M_1 & \rt \mathfrak{m}^{n+1}M_1:x_2/\mathfrak{m}^{n}M_1\\
	\rt \mathfrak{m}^{n+1}M_1/J\mathfrak{m}^nM_1
	&\rt \mathfrak{m}^{n+1}{M_2}/x_3\mathfrak{m}^n{M_2}\rt 0
	\end{align*}
	we get, if $b_1(x_2,M_1)=0$ then  $b_i(x_2,M_1)=0$ for all $i\geq2$.\\
	{\bf Claim:} depth$G(M_1)=1$.\\
	If possible assume that depth$G(M_1)=0$. This implies $ b_1(x_2,M_1) \geq 1$ (\ref{Property}).
	
	Consider $\overline{G(M_1)}=G(M_1)/(x_2^*,x_3^*)G(M_1)$, then we have
	$$G(M_1)/(x_2^*,x_3^*)G(M_1)=M_1/\mathfrak{m}M_1 \oplus \mathfrak{m}M_1/JM_1+\mathfrak{m}^2M_1\oplus \mathfrak{m}^2M_1/J\mathfrak{m}M_1+\mathfrak{m}^3M_1.$$
	Since red$M\leq2 $,  $\mathfrak{m}^3M_1\sub J\mathfrak{m}M_1$. So we have	$$\overline{G(M_1)}=M_1/\mathfrak{m}M_1 \oplus \mathfrak{m}M_1/JM_1+\mathfrak{m}^2M_1 \oplus \mathfrak{m}^2M_1/J\mathfrak{m}M_1$$
	Now from the above exact sequence we get
	$$0\rt \mathfrak{m}^2M_1:x_2/\mathfrak{m}M_1\rt \mathfrak{m}^2M_1/J\mathfrak{m}M_1\rt \mathfrak{m}^2M_2/x_3\mathfrak{m}M_2\rt 0.$$
	So we have $\ell(\mathfrak{m}^2M_1/J\mathfrak{m}M_1)=\rho_1+b_1(x_2,M_1)$.\\
	Hilbert series of $\overline{G(M_1)}$ is $3+\alpha z+(\rho_1+b_1(x_2,M_1))z^2$ where   $\rho_1+b_1(x_2,M_1)\geq4$ and $\alpha\leq 3$,  because $\ell(\mathfrak{m}M_1/JM_1)=4$ and $\mathfrak{m}^2M_1\not\sub JM_1$.   From [\ref{overline{G(M)}}] we know that this is not a possible Hilbert series.\\ So we have    depth$G(M_1)\geq 1$. Notice that $G(M_1)$ cannot be a \CM\ module, because in this case depth$G(M_2)=0$. This implies  depth$G(M_1)= 1$. Now  by Sally-descent  depth$G(M)=2$ and $h_M(z)=3+z+3z^2$.\\
	Now assume dim$M\geq4$ and $\underline{x}=x_1,\ldots,x_d$ a maximal $\phi$-superficial sequence. Set $M_{d-3}=M/(x_1,\ldots,x_{d-3})M$. We now have two cases.\\
	First case when $G(M_{d-3})$ is \CM. By Sally-descent $G(M)$ is \CM\ and $h_M(z)=3+2z+2z^2$.\\
	Second case when depth $G(M_{d-3})=2$. By Sally-descent depth$G(M)=d-1$ and $h_M(z)=3+z+3z^2$.

	{\bf Case(8):} $a_1=2,a_2=2,a_3=3$\\
	In this case  $M_d\cong Q'/(y^2)\oplus Q'/(y^2)\oplus Q'/(y^3)$ and $h_{M_d}(z)=3+3z+z^2$. \\
	We  first consider the case when  dim$M=3$ because if dim$M\leq2$ there is nothing to prove.\\
	Let $\underline{x}=x_1,x_2,x_3$ be a maximal $\phi$-superficial sequence. Set $M_1=M/x_1M$, $M_2=M/(x_1,x_2)M$ and $J=(x_2,x_3)$.\\ 
	Since dim$M_2=1$ we can write $h$-polynomial of $M_2$ as $h_{M_2}(z)=3+(\rho_0-\rho_1)z+\rho_1z^2$ where $\rho_n=\ell(\mathfrak{m}^{n+1}M_2/{x_3\mathfrak{m}^nM_2})$.
	In this case $\rho_0=4$ and since $i(M)=i(M_2)=i(M_3)=2$, this implies $im(\phi \otimes A/(x_1,x_2))\sub \mathfrak{m}^2(Q/(x_1,x_2)\oplus Q/(x_1,x_2))$. So we get $\ell({\mathfrak{m}M_2}/{\mathfrak{m}^2M_2})=3+(\rho_0-\rho_1)=6$, because $M_2=$coker$(\phi\otimes A/(x_1,x_2))$. This implies $\rho_1=1$ and so  $h_{M_2}=3+3z+z^2$ and so $G(M_2)$ is \CM. By Sally-descent $G(M)$ is \CM\ and $h_M(z)=3+3z+z^2$.\\ Now assume dim$M\geq4$ and $\underline{x}=x_1,\ldots,x_d$ a maximal $\phi$-superficial sequence. Set $M_{d-3}=M/(x_1,\ldots,x_{d-3})M$. So, $G(M_{d-3})$ is \CM.\\
	By Sally-descent $G(M)$ is \CM\ and $h_{M}(z)=3+3z+z^2$.
	
	{\bf Case(9):} $a_1=2,a_2=3,a_3=3$\\
	In this case  $M_d\cong Q'/(y^2)\oplus Q'/(y^3)\oplus Q'/(y^3)$ and $h_{M_d}(z)=3+3z+2z^2$. \\
	We first consider the case when dim$M=3$ because if dim$M\leq2$ there is nothing to prove.\\
	Let $\underline{x}=x_1,x_2,x_3$ be a maximal $\phi$-superficial sequence. Set $M_1=M/x_1M$, $M_2=M/(x_1,x_2)M$ and $J=(x_2,x_3)$.\\ 
	Since dim$M_2=1$, we can write $h$-polynomial of $M_2$ as $h_{M_2}(z)=3+(\rho_0-\rho_1)z+\rho_1z^2$ where $\rho_n=\ell(\mathfrak{m}^{n+1}M_2/{x_3\mathfrak{m}^nM_2})$.
	In this case $\rho_0=5$. Since $i(M)=i(M_2)=i(M_3)=2$,  this implies $im(\phi \otimes Q/(x_1,x_2))\sub \mathfrak{m}^2(Q/(x_1,x_2)\oplus Q/(x_1,x_2))$. So we get  $\ell({\mathfrak{m}M_2}/{\mathfrak{m}^2M_2})=3+(\rho_0-\rho_1)=6$, because $M_2=$coker$(\phi \otimes Q/(x_1,x_2))$. This implies $\rho_1=2$ and so $h_{M_2}=3+3z+2z^2$. Therefore,  $G(M_2)$ is \CM. By Sally-descent $G(M)$ is \CM\ and $h_M(z)=3+3z+2z^2$. \\ Now assume dim$M\geq4$ and $\underline{x}=x_1,\ldots,x_d$ a maximal $\phi$-superficial sequence. Set $M_{d-3}=M/(x_1,\ldots,x_{d-3})M$. So, $G(M_{d-3})$ is \CM.\\
	By Sally-descent $G(M)$ is \CM\ and $h_{M}(z)=3+3z+2z^2$.
	
	{\bf Case(10):} $a_1=3,a_2=3,a_3=3$\\
	In this case, $M_d\cong Q'/(y^3)\oplus Q'/(y^3)\oplus Q'/(y^3)$ and $h_{M_d}(z)=3+3z+3z^2$. So,  $e(M_d)=i(M_d)\mu(M_d)=9$. Now if dim$M\geq1$ then $e(M)=i(M)\mu(M)$, because this equality is preserved modulo any $\phi$-superficial sequence.  This implies $G(M)$  is \CM \ and $h_M(z)=3+3z+3z^2$ (see \cite[Theorem 2]{PuMCM}).
\end{proof}

It is clear from above theorem :
\begin{enumerate}
	\item If $a_1=1 ,a_2=1 ,a_3= 1$ then $G(M)$ is \CM\ and $h_M(z)=3$.  
	\item If $a_1=1 ,a_2=1 ,a_3=2 $ then :
	\begin{enumerate}
		\item $G(M)$ is \CM\ if and only if $h_M(z)=3+z$.
		\item depth$G(M)=d-1$ if and only if $h_M(z)=3+z^2$.
	\end{enumerate}
	\item If $a_1=1 ,a_2=1 ,a_3=3$ then $G(M)$ is \CM\ and $h_M(z)=3+z+z^2$.  
	\item If $a_1=1 ,a_2=2 ,a_3=2 $ then:
	
	\begin{enumerate}
		\item $G(M)$ is \CM\ if and only if $h_M(z)=3+2z$.
		\item depth$G(M)=d-1$ if and only if $h_M(z)=3+z+z^2$.
		\item depth$G(M)=d-2$ if and only if $h_M(z)=3+3z^2-z^3.$
	\end{enumerate}
	\item If $a_1= 1,a_2=2 ,a_3=3 $ then:
	\begin{enumerate}
		\item $G(M)$ is \CM \ if and only if $h_M(z)=3+2z+z^2$.
		\item depth$G(M)=d-1$ if and only if $h_M(z)=3+z+2z^2.$
	\end{enumerate}
	\item If $a_1=2 ,a_2=2 ,a_3=2 $ then $G(M)$ is \CM\ and $h_M(z)=3+3z$. 
	\item If $a_1=1 ,a_2=3 ,a_3= 3$ then:
	\begin{enumerate}
		\item $G(M)$ is \CM\ if and only if $h_M(z)=3+2z+2z^2$.
		\item depth$G(M)=d-1$ if and only if $h_M(z)=3+z+3z^2.$
	\end{enumerate}
	\item If $a_1=2 ,a_2=2 ,a_3=3 $ then $G(M)$ is \CM\ and $h_M(z)=3+3z+z^2$. 
	\item If $a_1=2 ,a_2=3 ,a_3=3 $ then $G(M)$ is \CM\ and $h_M(z)=3+3z+2z^2$. 
	\item If $a_1= 3,a_2=3 ,a_3= 3$ then $G(M)$ is \CM\ and $h_M(z)=3+3z+3z^2$. 
\end{enumerate}
{\bf Note:} red$(M)< 2$ occurs only in cases (1), (2a), (4a), (6).\\
Now we can conclude:
\begin{corollary}
	Let $(A,\mathfrak{m})$ be a hypersurface ring  of dimension $d\geq 3$ with infinite residue field and  $M$ an MCM $A$-module with $\mu(M)=3$.
	Now if  $red(M)=2$, then depth$G(M)\geq d-2$. Also, possible $h$-polynomials are 
	
	\[ h_M(z) = \begin{cases} 
	3+3z^2-z^3 & depth$G(M)=d-2$. \\
	3+z^2 & depth$G(M)=d-1$. \\
	3+z+z^2 & depth$G(M)=d-1$. \\
	3+z+2z^2 & depth$G(M)=d-1$. \\
	3+z+3z^2 & depth$G(M)=d-1$. \\
	3+z+z^2 & G(M) \ \text{is Cohen-Macaulay}. \\
	3+2z+z^2 &  G(M) \ \text{is Cohen-Macaulay}. \\
	3+2z+2z^2 &  G(M) \ \text{is Cohen-Macaulay}. \\
	3+3z+z^2 &  G(M) \ \text{is Cohen-Macaulay}. \\
	3+3z+2z^2 &  G(M) \ \text{is Cohen-Macaulay}.\\
	3+3z+3z^2 &  G(M) \ \text{is Cohen-Macaulay}.
	\end{cases}
	\]
\end{corollary}

\section{\bf  The case when $\mu(M)=4$ and $e(A)=3$}

We first prove a lemma.

\begin{lemma}\label{le1}
	Let $(A,\mathfrak{m})$ be a complete hypersurface ring of dimension $d$ with $e(A)=3$ and $M$ be a MCM module. Let $\underline{x}=x_1,\ldots,x_d$ be a maximal $\phi$-superficial sequence. If $M$ has no free summand, then $M_d=M/(x_1,\ldots,x_d)M$ also has no free summand.
\end{lemma}
\begin{proof}
	Since $M$ has no free summand there exists an MCM module $L$ such that $M=Syz_1^{A}(L)$ ( for instance see \cite[Theorem 6.1]{Eisenbud}). So we have  $0\rt M\rt F \rt L\rt 0$ where $F=A^{\mu(L)}$ and $M\sub \mathfrak{m}F$. Set $F_d=F/(\underline{x})F$. Going modulo $\underline{x}$ we get $M_d\sub \mathfrak{m}F_d$. So we have $\mathfrak{m}^2M_d\sub \mathfrak{m}^3F_d=0$, because red$(A)=2$. This implies $M_d$ has no free summand.
\end{proof}
\begin{theorem}
	Let $(A,\mathfrak{m})$ be a  hypersurface ring of dimension $d$ with $e(A)=3$ and $M$ an MCM module with no free summand. If  $\mu(M)=4$, then depth$G(M)\geq d-3$.
\end{theorem}
\begin{proof}
	By \ref{Base change} we can assume that $A$ is a complete local ring with infinite residue field.
	Since $e(A)=3$,  we can take $(A,\mathfrak{m})=(Q/(g),\mathfrak{n}/(g))$ where $(Q,\mathfrak{n})$ is a regular local ring of dimension $d+1$ and $g\in \mathfrak{n}^3\setminus\mathfrak{n}^4$. This implies that $h_A(z)=1+z+z^2$ and $\mathfrak{m}^3=J\mathfrak{m}^2$.   Since $\mathfrak{m}^3=J\mathfrak{m}^2$,  we have red$(M)\leq 2$.
	
	Let dim$M\geq1$ and $0\rt Q^4\xrightarrow{\phi} Q^4\rt M\rt 0$ be a minimal presentation of $M$. Let $\underline{x}=x_1,\ldots,x_d$ be a maximal $\phi$-superficial sequence (see \ref{phi}). Set $M_d=M/\underline{x}M$ and $(Q',(y))=(Q/(\underline{x}),\mathfrak{n}/(\underline{x}))$.
	
	Clearly, $Q'$ is DVR and so $M_d\cong Q'/(y^{a_1})\oplus Q'/(y^{a_2})\oplus Q'/(y^{a_3})\oplus Q'/(y^{a_4})$.
	
	From the Lemma \ref{le1}, $M_d$ has no free summand. So, we can assume that $1\leq a_1\leq a_2\leq a_3\leq a_4\leq 2$.

	This implies $4\leq e(M)\leq 8$. Now we  consider all cases separately:
	
	{\bf Case(1): $e(M)=4$}\\ 
	In this case $M_d\cong Q'/(y)\oplus Q'/(y)\oplus Q'/(y)\oplus Q'/(y)$. This implies $h_{M_d}=4$, so $e(M_d)=\mu(M_d)=4$.\\ For dim$M\geq1$, $e(M)=\mu(M)=4$. Also, notice that $i(M)=i(M_d)=1$. Since $e(M)=\mu(M)=4$, $M$ is an Ulrich module (see \ref{Ulrich}).  This implies that $G(M) $ is Cohen-Macaulay and $h_M(z)=4$ (see \cite[Theorem 2]{PuMCM}). 
	
	{\bf Case(2): $e(M)=5$}\\
	In this case   $M_d\cong Q'/(y)\oplus Q'/(y)\oplus Q'/(y)\oplus Q'/(y^2)$ and this implies  $h_{M_d}(z)=4+z$.\\
	We first consider the case when dim$M=4$ because if dim$M\leq3$ there is nothing to prove.
	Let $\underline{x}=x_1,x_2,x_3,x_4$ be a maximal $\phi$-superficial sequence. Set $M_1=M/x_1M$, $M_2=M/(x_1,x_2)M$, $M_3=M/(x_1,x_2,x_3)M$ and $J=(x_3,x_4)$.\\ 
	Since dim$M_3=1$ we can write $h$-polynomial of $M_3$ as $h_{M_3}(z)=4+(\rho_0-\rho_1)z+\rho_1z^2$
	where $\rho_n=\ell(\mathfrak{m}^{n+1}M_3/{x_4\mathfrak{m}^nM_3})$. Since $\rho_0=1$ and all the coefficients of $h_{M_3}$ are non-negative [see \ref{d=1}(1)],  possible values of $\rho_1$ are $0$ and $1$.\\  
	{\bf Subcase (i): $\rho_1=0.$}\\ In this case, $M_3$ has minimal multiplicity and $h_{M_3}(z)=4+z$. So $G(M_3)$ is \CM. By Sally-descent $G(M)$ is \CM.\\
	{\bf Subcase (ii): $\rho_1=1$.}\\
	In this case,  $h_{M_3}(z)=4+z^2$. This implies depth$G(M_3)=0$, because $h_{M_3}(z)\ne h_{M_4}(z)$ (see \ref{Property}).\\
	Since dim$M_2=2$, we have (see \ref{mod-sup})  $$e_2(M_2)=e_2({M_3})-\sum b_i(x_3,M_2),$$
	where $b_i(x_3,M_2)=\ell(\mathfrak{m}^{i+1}M_2:x_3/\mathfrak{m}^iM_2)$.
	We know that $e_2(M_2)$  and $\sum b_i(x_3,M_2) $ are non-negative integers and  $e_2({M_3})=1$. This implies  $\sum b_i(x_3,M_2) \leq1$.
	
	Since red$(M)\leq2$, from exact sequence (\ref{exact seq})
	\begin{align*}
	0 \rt \mathfrak{m}^nM_2:J/\mathfrak{m}^{n-1}M_2\rt \mathfrak{m}^nM_2:x_3/\mathfrak{m}^{n-1}M_2 & \rt \mathfrak{m}^{n+1}M_2:x_3/\mathfrak{m}^{n}M_2\\
	\rt \mathfrak{m}^{n+1}M_2/J\mathfrak{m}^nM_2
	&\rt \mathfrak{m}^{n+1}{M_3}/x_4\mathfrak{m}^n{M_3}\rt 0
	\end{align*}
	we get, if $b_1(x_3,M_2)=0$ then  $b_i(x_3,M_2)=0$ for all $i\geq2$.\\
	{\bf Claim:} depth$G(M_2)=1$.\\
	{\bf Proof of the claim:}
	If possible assume that depth$G(M_2)=0$. This implies $ b_1(x_3,M_2) = 1$ (see \ref{Property}). 
	
	From the above exact sequence we have
	$$ 0\rt \mathfrak{m}^{2}M_2:x_3/\mathfrak{m}M_2
	\rt \mathfrak{m}^{2}M_2/J\mathfrak{m}M_2
	\rt \mathfrak{m}^{2}{M_3}/x_4\mathfrak{m}{M_3}\rt 0.$$
	So, $\ell(\mathfrak{m}^{2}M_2/J\mathfrak{m}M_2)=\rho_1+b_1(x_3,M_2)$.
	
	Consider $\overline{G(M_2)}=G(M_2)/(x_3^*,x_4^*)G(M_2)$ then we have
	$$G(M_2)/(x_3^*,x_4^*)G(M_2)=M_2/\mathfrak{m}M_2 \oplus \mathfrak{m}M_2/JM_2+\mathfrak{m}^2M_2\oplus \mathfrak{m}^2M_2/J\mathfrak{m}M_2+\mathfrak{m}^3M_2$$
	Thus we have 	$$\overline{G(M_2)}=M_2/\mathfrak{m}M_2 \oplus \mathfrak{m}M_2/JM_2 \oplus \mathfrak{m}^2M_2/J\mathfrak{m}M_2.$$ Its Hilbert series is
	$4+z+(\rho_1+b_1(x_3,M_2))z^2$ where $\rho_1+b_1(x_3,M_2)=2$. But this is not a possible Hilbert series (see \ref{overline{G(M)}}). So  this implies that depth$G(M_2)\geq 1$. Notice that depth$G(M_2)\ne 2$, because depth$G(M_3)=0$ in this case. Therefore depth$G(M_2)=1$. By Sally-descent depth$G(M)=3$.
	
	Now assume dim$M\geq5$ and $\underline{x}=x_1,\ldots,x_d$ a maximal $\phi$-superficial sequence. Set $M_{d-4}=M/(x_1,\ldots,x_{d-4})M$. We now have two cases.\\
	First case when $G(M_{d-4})$ is \CM. By Sally-descent $G(M)$ is \CM\ and $h_M(z)=4+z$.\\
	Second case when depth $G(M_{d-4})=3$. By Sally-descent depth$G(M)=d-1$ and $h_M(z)=4+z^2$.

	{\bf Case(3): $e(M)=6$}\\
	In this case  $M_d\cong Q'/(y)\oplus Q'/(y)\oplus Q'/(y^2)\oplus Q'/(y^2)$ and this implies  $h_{M_d}(z)=4+2z$.\\
	We first consider the  case when dim$M=4$ because if dim$M\leq3$ there is nothing to prove.\\
	Let $\underline{x}=x_1,x_2,x_3,x_4$ be a maximal $\phi$-superficial sequence. Set $M_1=M/x_1M$, $M_2=M/(x_1,x_2)M$, $M_3=M/(x_1,x_2,x_3)M$, $J_1=(x_2,x_3,x_4)$  and $J_2=(x_3,x_4)$.\\ 
	Since dim$M_3=1$ we can write $h$-polynomial of $M_3$ as $h_{M_3}(z)=4+(\rho_0-\rho_1)z+\rho_1z^2$
	where $\rho_n=\ell(\mathfrak{m}^{n+1}M_3/{x_4\mathfrak{m}^nM_3})$.
	So we have $\rho_0=2$ and since all the coefficients of $h_M$ are non-negative [see \ref{d=1}(1)], so possible values of $\rho_1$ are $0$, $1$ and $2$.\\
	{\bf Subcase(i): $\rho_1=0$.}\\
	In this case, $M_3$ has minimal multiplicity and $h_{M_3}(z)=4+2z$. This implies  $G(M_3)$ is \CM. By Sally-descent $G(M)$ is \CM.\\
	{\bf Subcase(ii): $\rho_1=1$. }\\
	In this case,  $h_{M_3}(z)=4+z+z^2$. This implies  depth$G(M_3)=0$, because $h_{M_3}(z)\ne h_{M_4}(z)$ (see \ref{Property}).\\
	Since dim$M_2=2$, we have (see \ref{mod-sup})
	$$e_2(M_2)=e_2({M_3})-\sum b_i(x_3,M_2),$$
	where $b_i(x_3,M_2)=\ell(\mathfrak{m}^{i+1}M_2:x_3/\mathfrak{m}^iM_2)$.
	We know that $e_2(M_2) $ and $\sum b_i(x_3,M_2)  $ are non-negative integers. In this case we also have  $e_2({M_3})=1$. This implies $\sum b_i(x_3,M_2) \leq 1 $.
	
	Since red$(M)\leq2$, from exact sequence 
	\begin{align*}
	0 \rt \mathfrak{m}^nM_2:J_2/\mathfrak{m}^{n-1}M_2\rt \mathfrak{m}^nM_2:x_3/\mathfrak{m}^{n-1}M_2
	& \rt \mathfrak{m}^{n+1}M_2:x_3/\mathfrak{m}^{n}M_2\\
	\rt \mathfrak{m}^{n+1}M_2/J_2\mathfrak{m}^nM_2
	&\rt \mathfrak{m}^{n+1}{M_3}/x_4\mathfrak{m}^n{M_3}\rt 0
	\end{align*}
	we get, if $b_1(x_3,M_2)=0$ then  $b_i(x_3,M_2)=0$ for all $i\geq2$. So, $\sum b_i(x_3,M_2)  \ne 0$ implies $b_1(x_3,M_2)=1$. \\
	Now we have two cases.\\
	{\bf Subcase (ii).(a):} When $b_1(x_3,M_2)=0$.\\ This implies  depth$G(M_2)\geq 1$ (see \ref{Property}). In fact depth$G(M_2)=1$ otherwise $G(M_2)$ is \CM. This  is not possible because depth$G(M_3)=0$. Now by Sally-descent depth$G(M)=3$ and $h_M(z)=4+z+z^2$. \\
	{\bf Subcase (ii).(b):} When  $b_1(x_3,M_2)=1$.\\ So, depth$G(M_2)=0$ (see \ref{Property}). In this case $h_{M_2}(z)=h_{M_3}(z)-(1-z)^2z=4+3z^2-z^3$ (see \ref{mod-sup}).
	
	From the above exact sequence we get
	\begin{align} \label{M_2}
	0\rt \mathfrak{m}^{2}M_2:x_3/\mathfrak{m}M_2
	\rt \mathfrak{m}^{2}M_2/J_2\mathfrak{m}M_2
	\rt \mathfrak{m}^{2}{M_3}/x_4\mathfrak{m}{M_3}\rt 0.
	\end{align}
	
	So, $\ell(\mathfrak{m}^{2}M_2/J_2\mathfrak{m}M_2)=\rho_1+b_1(x_3,M_2)=2$.\\
	Since dim$M_1=3$,  from short exact sequence (see \ref{exact d}) $$0\rt \mathfrak{m}^2M_1:x_2/\mathfrak{m}M_1\rt {\mathfrak{m}^2M_1}/J_1\mathfrak{m}M_1\rt \mathfrak{m}^2M_2/{J_2}\mathfrak{m}M_2\rt 0$$
	we have $\ell(\mathfrak{m}^2M_1/J_1\mathfrak{m}M_1)=2+\ell(\mathfrak{m}^2M_1:x_2/\mathfrak{m}M_1)$.
	We also have
		$$G(M_1)/(x_2^*,x_3^*,x_4^*,)G(M_1)=M_1/\mathfrak{m}M_1 \oplus \mathfrak{m}M_1/J_1M_1\oplus \mathfrak{m}^2M_1/J_1\mathfrak{m}M_1.$$
	By considering its Hilbert series we get   $\ell(\mathfrak{m}^2M_1/J_1\mathfrak{m}M_1)\leq 2$, because in this case $\ell(\mathfrak{m}M_1/J_1M_1)=2$ (see \ref{overline{G(M)}}). Therefore we have $\ell(\mathfrak{m}^2M_1/J_1\mathfrak{m}M_1)= 2$.
	
	We also know that $\mathfrak{m}^2M_1\sub J_1M_1$.\\
	So in this case 
	$$\delta =\sum\ell(\mathfrak{m}^{n+1}M_1\cap J_1M_1/J_1\mathfrak{m}^nM_1)= 2$$
	
	We know that if $\delta\leq2$ then depth$G(M)\geq d-\delta$ (see \cite[Theorem 5.1]{apprx}). So we have depth$G(M_1)\geq 1$. Also notice that depth$G(M_1)=1$, because depth$G(M_2)=0$.\\ 
	By Sally-descent  depth$G(M)=2$ and $h_M(z)=4+3z^2-z^3$. \\
	{\bf Subcase(iii): $\rho_1=2$}\\
	In this case $h_{M_3}(z)=4+2z^2$. This implies depth$G(M_3)=0$, because $h_{M_3}(z)\ne h_{M_4}(z) $ (see \ref{Property}).\\
	Since dim$M_2=2$, we have (see \ref{Property})
	$$e_2(M_2)=e_2({M_3})-\sum b_i(x_3,M_2).$$
	We know that $e_2(M_2) $ and $\sum b_i(x_3,M_2)  $ are non-negative integers. In this case we also have  $e_2({M_3})=2$. This implies $\sum b_i(x_3,M_2) \leq 2$.\\
	We know that  $b_1(x_3,M_2)=0$ implies all $b_i(x_3,M_2)=0$.\\
	Now we have two cases.\\
	{\bf Subcase (iii).(a):} When $b_1(x_3,M_2)=0$.\\
	So, in this case depth$G(M_2)=1$ (see \ref{Property}). Also notice depth$G(M_2)\ne 2$ because depth$G(M_3)=0$. By Sally-descent depth$G(M)=3$ and $h_M(z)=4+2z^2$.\\  
	{\bf Subcase (iii).(b):} When $b_1(x_3,M_2)\neq0$.\\ Now from the exact sequence (\ref{M_2}) we get $\ell(\mathfrak{m}^2M_2/J_2\mathfrak{m}M_2)=\rho_1+b_1(x_3,M_2)\geq3$.\\  Now consider  $\overline{G(M_2)}=G(M_2)/(x_3^*,x_4^*)G(M_2)$. Then we get 
	$$\overline{G(M_2)}=M_2/\mathfrak{m}M_2\oplus \mathfrak{m}M_2/J_2M_2\oplus\mathfrak{m}^2M_2/J_2\mathfrak{m}M_2$$ 
	Its  Hilbert series is $4+2 z+(\rho_1+b_1(x_3,M_2))$,  because $\ell(\mathfrak{m}M_2/J_2M_2)=2$. But this is not  a possible Hilbert series (see \ref{overline{G(M)}}). Therefore  the case when $b_1(x_3,M_2)\neq0$ is not possible.
	
	Now assume dim$M\geq5$ and $\underline{x}=x_1,\ldots,x_d$ a maximal $\phi$-superficial sequence. Set $M_{d-4}=M/(x_1,\ldots,x_{d-4})M$. We now have two cases.\\
	First case when $G(M_{d-4})$ is \CM. By Sally-descent $G(M)$ is \CM\ and $h_M(z)=4+2z$.\\
	Second case when depth $G(M_{d-4})=3$. By Sally-descent depth$G(M)=d-1$ and $h_M(z)=4+z+z^2$ or $h_M(z)=4+2z^2$.\\
	Third case when depth $G(M_{d-4})=2$. By Sally-descent depth$G(M)=d-2$ and $h_M(z)=4+3z^2-z^3$.
	%Fourth case when depth $G(M_{d-4})=3$. By Sally-descent depth$G(M)=d-1$ and $h_M(z)=4+2z^2$.\\

	{\bf Case(4): $e(M)=7$}\\
	In this case  $M_d\cong Q'/(y)\oplus Q'/(y^2)\oplus Q'/(y^2)\oplus Q'/(y^2)$ and $h_{M_d}(z)=4+3z$.\\
	We first consider the case when dim$M=4$ because if dim$M\leq3$ there is nothing to prove.\\
	Let $\underline{x}=x_1,x_2,x_3,x_4$ be a maximal $\phi$-superficial sequence. Set $M_1=M/x_1M$, $M_2=M/(x_1,x_2)M$, $M_3=M/(x_1,x_2,x_3)M$, $J_1=(x_2,x_3,x_4)$, $J_2=(x_3,x_4)$ and $J=(x_1,x_2,x_3,x_4)$.\\ 
	We first prove two claims:
	
	{\bf Claim(1):} $\widetilde{\mathfrak{m}^iM_3}=\mathfrak{m}^iM_3$ for all $i\geq2.$\\
	{\bf Proof of Claim:} Since we have $\mathfrak{m}^{n+1}M_3=x_4\mathfrak{m}^nM_3$ for all $n\geq2$. So $(\mathfrak{m}^{n+1}M_3:x_4)=\mathfrak{m}^nM_3$ for all $n\geq2$. We have  exact sequence (see \ref{RR-2}) $$0\rt \mathfrak{m}^{n+1}M_3:x_4/\mathfrak{m}^nM_3\rt \widetilde{\mathfrak{m}^nM_3}/\mathfrak{m}^nM_3\rt \widetilde{\mathfrak{m}^{n+1}M_3}/\mathfrak{m}^{n+1}M_3.$$ We also know that for $n\ggg 0$, $\widetilde{\mathfrak{m}^nM_3}=\mathfrak{m}^nM_3$. By using these facts it is clear that $\widetilde{\mathfrak{m}^iM_3}=\mathfrak{m}^iM_3$ for all $i\geq2.$
	
	{\bf Claim(2):} $\ell(\widetilde{\mathfrak{m}M_3}/\mathfrak{m}M_3)\leq1$.\\
	{\bf Proof of the claim:} Since $\mu(M_3)=4$,  we have $\ell(\widetilde{\mathfrak{m}M_3}/\mathfrak{m}M_3)\leq4$.\\ If $\ell(\widetilde{\mathfrak{m}M_3}/\mathfrak{m}M_3)=4$ then $\widetilde{\mathfrak{m}M_3}=M_3$. So $\mathfrak{m}^3M_3=\mathfrak{m}^2M_3$ because we know that $\widetilde{\mathfrak{m}^iM_3}=\mathfrak{m}^iM_3$ for all $i\geq2$. So, from here we have $\mathfrak{m}^2M_3=0$ which is a contradiction. Therefore $\ell(\widetilde{\mathfrak{m}M_3}/\mathfrak{m}M_3)\leq3$.\\
	If possible assume that $\ell(\widetilde{\mathfrak{m}M_3}/\mathfrak{m}M_3)>1$, so we have $M_3=\langle m_1,m_2,l_1,l_2 \rangle$ where $l_1,l_2\in \widetilde{\mathfrak{m}M_3}\setminus\mathfrak{m}M_3$. This implies $l_i\mathfrak{m}\sub \widetilde{\m^2M_3}=\mathfrak{m}^2M_3$ for $i=1,2$. Now if we set $\mathfrak{m}'=\mathfrak{m}/(x_1,x_2,x_3,x_4)$ then $\mathfrak{m}'$ is a principal ideal. We also know that $\ell(\mathfrak{m}M_4)=\ell(\mathfrak{m}'M_4)$ and $\mathfrak{m}^2M_4=\mathfrak{m'}^2M_4=0$. From here  we get  $\ell({\mathfrak{m}M_4})\leq2$. This is a contradiction because we know that $\ell({\mathfrak{m}M_4})=3$. So, $\ell(\widetilde{\mathfrak{m}M_3}/\mathfrak{m}M_3)\leq1$.\\
	Now we have two cases.\\
	{\bf Subcase (i):} When $\widetilde{\mathfrak{m}M_3}=\mathfrak{m}M_3$.\\ So, we have $\widetilde{\mathfrak{m}^iM_3}=\mathfrak{m}^iM_3$ for all $i$, because we know that $\widetilde{\mathfrak{m}^iM_3}=\mathfrak{m}^iM_3$ for all $i\geq2$ (from claim(1)). So in this case depth$G(M_3)=1$, i.e. $G(M_3)$ is \CM. By Sally-descent $G(M)$ is \CM\ and $h_M(z)=4+3z$.\\
	{\bf Subcase (ii):}  When $\ell(\widetilde{\mathfrak{m}M_3}/\mathfrak{m}M_3)=1$.\\
	Since dim$M_3=1$ we can write $h$-polynomial of $M_3$ as $h_{M_3}(z)=4+(\rho_0-\rho_1)z+\rho_1z^2$ where $\rho_n=\ell(\mathfrak{m}^{n+1}M_3/{x_4\mathfrak{m}^nM_3})$. We have  $\rho_0=\ell(\mathfrak{m}M_2/x_3M_2)=3$ and coefficients of $h_{M_3}$ are non-negative [from \ref{d=1}(1)]. 
	
	From short exact sequence (see \ref{exact d})
	$$0\rt \mathfrak{m}^2M_3:x_4/\mathfrak{m}M_3\rt \mathfrak{m}^2M_3/x_4\mathfrak{m}M_3\rt \mathfrak{m}^2M_4/0\rt 0$$
	we have $\rho_1=b_1(x_4,M_3)$ because $\mathfrak{m}^2M_4=0$.  From Claim(2) we have $b_1(x_4,M_3)\leq1$.
	
	We now have two cases.\\
	{\bf Subcase (ii).(a):} When $\rho_1=b_1(x_4,M_3)=0$.\\ So, $M_3$ has minimal multiplicity. This implies $G(M_3)$ is \CM. So this not a possible subcase because in subcase (ii),  $\widetilde{\mathfrak{m}M_3}\ne \mathfrak{m}M_3$.\\
	{\bf Subcase (ii).(b):} When $\rho_1=b_1(x_4,M_3)=\ell(\mathfrak{m}^2M_3:x_4/\mathfrak{m}M_3)=1$.\\ In this case depth$G(M_3)=0$, because $h_{M_3}(z)\ne h_{M_4}(z)$ (see \ref{Property}). We have  $h_{M_3}(z)=4+2z+z^2.$\\
	We also have $\widetilde{\mathfrak{m}^2M_3}=x_4\widetilde{\mathfrak{m}M_3}$. In fact, if $a\in\widetilde{\mathfrak{m}^2M_3}  $ then we can write $a=xp$ because $\widetilde{\mathfrak{m}^2M_3}=\mathfrak{m}^2M_3\sub (x_4)M_3$. This implies that $p\in \widetilde{\mathfrak{m}M_3}$ because $(\widetilde{\mathfrak{m}^2M_3}:x_4)=\widetilde{\mathfrak{m}M_3}$. So we have $\widetilde{\mathfrak{m}^{i+1}M_3}=x_4\widetilde{\mathfrak{m}^iM_3}$ for all $i\geq 1$ because $\mathfrak{m}^{i+1}M_3=x_4\mathfrak{m}^iM_3$ for $i\geq2$. This implies that $\widetilde{G(M_3)}$ has minimal multiplicity and $\widetilde{h_{M_3}}(z)=3+4z$.\\
	Here we have two cases.\\
	{\bf Subcase (ii).(b).(1):} When depth$G(M_2)\ne 0$.\\ Then we have depth$G(M_2)=1$ because depth$G(M_3)=0$. By Sally-descent depth$G(M)=3$ and $h_M(z)=4+2z+z^2$.\\
	{\bf Subcase (ii).(b).(2): } When depth$G(M_2)=0$.\\ So we have $b_1(x_3,M_2)=\ell(\mathfrak{m}^2M_2:x_3/\mathfrak{m}M_2)\neq 0$ (see \ref{Property}).\\
	From exact sequences (see \ref{RR-2})
	$$0\rt \frac{\mathfrak{m}^{n+1}M_2:x_3}{\mathfrak{m}^nM_2}\rt \frac{\widetilde{\mathfrak{m}^nM_2}}{\mathfrak{m}^nM_2}\rt \frac{\widetilde{\mathfrak{m}^{n+1}M_2}}{\mathfrak{m}^{n+1}M_2}\rt \frac{\widetilde{\mathfrak{m}^{n+1}M_3}}{\mathfrak{m}^{n+1}M_3} $$
	and $$0 \rt \widetilde{\mathfrak{m}M_2}/\mathfrak{m}M_2\rt \widetilde{\mathfrak{m}M_3}/\mathfrak{m}M_3 $$
	we have  $$1\leq b_1(x_3,M_2)\leq \ell(\widetilde{\mathfrak{m}M_2}/\mathfrak{m}M_2)\leq \ell(\widetilde{\mathfrak{m}M_3}/\mathfrak{m}M_3)=1. $$
	This implies that $\ell(\widetilde{\mathfrak{m}M_2}/\mathfrak{m}M_2)=1$ and $b_1(x_3,M_2)=1$. From here we also get 	 $\widetilde{\mathfrak{m}^nM_2}=\mathfrak{m}^nM_2$ for all $n\geq2$,
	because from claim(1) we know that $\widetilde{\mathfrak{m}^nM_3}=\mathfrak{m}^nM_3$ for all $n\geq2$.
	\\
	From the exact sequence (see \ref{exact d})
	$$ 0\rt \mathfrak{m}^{2}M_2:x_3/\mathfrak{m}M_2
	\rt \mathfrak{m}^{2}M_2/J_2\mathfrak{m}M_2
	\rt \mathfrak{m}^{2}{M_3}/x_4\mathfrak{m}{M_3}\rt 0$$
	we get $\ell(\mathfrak{m}^{2}M_2/J_2\mathfrak{m}M_2)=\rho_1+b_1(x_3,M_2)=2$.\\
	In this case we also have	
	$h_{M_2}(z)=h_{M_3}(z)-(1-z)^2z=4+z+3z^2-z^3$ (see \ref{mod-sup}).\\
	Since we have  $\ell(\widetilde{\mathfrak{m}M_2}/\mathfrak{m}M_2)=\ell(\widetilde{\mathfrak{m}M_3}/\mathfrak{m}M_3) $ and $\widetilde{\mathfrak{m}^iM_3}=\mathfrak{m}^iM_3$ for all $i\geq2$, from \cite[2.1]{Pu2}
	we get $$\overline{\widetilde{\mathfrak{m}^iM_2}}=\widetilde{\mathfrak{m}^iM_3}\ \ \text{ for all }  i\geq 1.$$	
	
	So, $\widetilde{G(M_2)}/x_3^*\widetilde{G(M_2)}=\widetilde{G(M_3)}$, this implies that $\widetilde{G(M_2)}$ is \CM \ and $\widetilde{h_{M_2}}(z)=3+4z$.\\
	From exact sequence  (see \ref{exact d})
	\begin{align}\label{M_1}
	0\rt \mathfrak{m}^{2}M_1:x_2/\mathfrak{m}M_1
	\rt \mathfrak{m}^{2}M_1/J_1\mathfrak{m}M_1
	\rt \mathfrak{m}^{2}{M_2}/J_2\mathfrak{m}{M_2}\rt 0
	\end{align}

	we have , if $\ell(\mathfrak{m}^2M_1:x_2/\mathfrak{m}M_1)=0$ then
	
	$\ell(\mathfrak{m}^2M_1/J_1\mathfrak{m}M_1)=2$, because $\ell(\mathfrak{m}^2M_2/J_2\mathfrak{m}M_2)=2$.\\
	{\bf Subcase (ii).(b).(2).($\alpha$):} $\ell(\mathfrak{m}^2M_1:x_2/\mathfrak{m}M_1)=0$.\\  Consider $$\delta=\sum\ell(\mathfrak{m}^{n+1}M_1\cap J_1M_1/J_1\mathfrak{m}^nM_1).$$ 
	We know that if $\delta\leq2$ then depth$G(M_1)\geq d-\delta$ (see \cite[Theorem 5.1]{apprx}).
	Since  $\delta=2$,  depth$G(M_1)\geq1$. Notice that here depth$G(M_1)=1$, because depth$G(M_2)=0$. By Sally-descent depth$G(M)=2$ and $h_M(z)= 4+z+3z^2-z^3$.\\
	{\bf Subcase (ii).(b).(2).($\beta$):} $\ell(\mathfrak{m}^2M_1:x_2/\mathfrak{m}M_1)\neq0$.\\ This implies depth$G(M_1)=0$ (see \ref{Property}). 
	
	Now from exact sequences (see \ref{RR-2})
	$$0\rt \frac{\mathfrak{m}^{n+1}M_1:x_2}{\mathfrak{m}^nM_1}\rt \frac{\widetilde{\mathfrak{m}^nM_1}}{\mathfrak{m}^nM_1}\rt \frac{\widetilde{\mathfrak{m}^{n+1}M_1}}{\mathfrak{m}^{n+1}M_1}\rt \frac{\widetilde{\mathfrak{m}^{n+1}M_2}}{\mathfrak{m}^{n+1}M_2} $$
	and $$0 \rt \widetilde{\mathfrak{m}M_1}/\mathfrak{m}M_1\rt \widetilde{\mathfrak{m}M_2}/\mathfrak{m}M_2 $$
	
	we get $$1\leq \ell(\mathfrak{m}^2M_1:x_2/\mathfrak{m}M_1)\leq \ell(\widetilde{\mathfrak{m}M_1}/\mathfrak{m}M_1)\leq \ell(\widetilde{\mathfrak{m}M_2}/\mathfrak{m}M_2)=1. $$
	This implies $\ell(\mathfrak{m}^2M_1:x_2/\mathfrak{m}M_1)=\ell(\widetilde{\mathfrak{m}M_1}/\mathfrak{m}M_1)=1$. From here we also get  $\widetilde{\mathfrak{m}^iM_1}=\mathfrak{m}^iM_1$ for all $i\geq2$, because $\widetilde{\mathfrak{m}^iM_2}=\mathfrak{m}^iM_2$ for all $i\geq2$. \\  From the short exact sequence (\ref{M_1})
	we have $\ell(\mathfrak{m}^2M_1/J_1\mathfrak{m}M_1)=3$.\\ Now since $\ell(\widetilde{\mathfrak{m}M_1}/\mathfrak{m}M_1)=\ell(\widetilde{\mathfrak{m}M_2}/\mathfrak{m}M_2)$ and $\widetilde{\mathfrak{m}^iM_1}=\mathfrak{m}^iM_1$ for all $i\geq2$, from \cite[2.1]{Pu2} we get $$\overline{\widetilde{\mathfrak{m}^iM_1}}=\widetilde{\mathfrak{m}^iM_2}\ \ \text{for all } \ i\geq 1.$$
	
	So $\widetilde{G(M_1)}/x_2^*\widetilde{G(M_1)}=\widetilde{G(M_2)}$, this implies that $\widetilde{G(M_1)}$ is \CM \ and $\widetilde{h_{M_1}}(z)=3+4z$.\\
	We can write the $h$-polynomial of $M_1$ as  $h_{M_1}(z)=\widetilde{h_{M_1}}(z)+(1-z)^4$.\\
	Consider  $$G(M)/(x_1^*x_2^*,x_3^*,x_4^*)G(M)=M/\mathfrak{m}M\oplus \mathfrak{m}M/JM\oplus \mathfrak{m}^2M/J\mathfrak{m}M.$$ After looking at its Hilbert series we get $\ell(\mathfrak{m}^2M/J\mathfrak{m}M)\leq3$, because $\ell(\mathfrak{m}M/JM)=3$ (see \ref{overline{G(M)}}).
	
	We have short exact sequence (see \ref{exact d})
	$$0\rt \mathfrak{m}^2M:x_1/\mathfrak{m}M\rt \mathfrak{m}^2M/J\mathfrak{m}M\rt \mathfrak{m}^2M_1/{J_1}\mathfrak{m}M_1\rt 0.$$
	This implies that $\ell(\mathfrak{m}^2M/J\mathfrak{m}M)\geq3$.\\ So we have $$(\mathfrak{m}^2M:x_1)=\mathfrak{m}M \ \text{and}\ \ell(\mathfrak{m}^2M/J\mathfrak{m}M)=3. $$
	Now we first prove a claim.\\
	{\bf Claim:} $\widetilde{\mathfrak{m}M}=\mathfrak{m}M$.\\
	{\bf Proof of the claim:} 	If $\widetilde{\mathfrak{m}M}\neq \mathfrak{m}M$.
	From exact sequence (see \ref{RR-2})
	$$0 \rt \widetilde{\mathfrak{m}M}/\mathfrak{m}M\rt \widetilde{\mathfrak{m}M_1}/\mathfrak{m}M_1 $$
	we get $\ell(\widetilde{\mathfrak{m}M}/\mathfrak{m}M)=1$, because $\ell(\widetilde{\mathfrak{m}M}/\mathfrak{m}M)\leq \ell(\widetilde{\mathfrak{m}M_1}/\mathfrak{m}M_1)=1$. Since  in this case    
	$\ell(\widetilde{\mathfrak{m}M}/\mathfrak{m}M)= \ell(\widetilde{\mathfrak{m}M_1}/\mathfrak{m}M_1)$ and $\widetilde{\mathfrak{m}^iM_1}={\mathfrak{m}^iM_1}$ for all $i\geq 2$, from \cite[2.1]{Pu2} we get
	$$\overline{\widetilde{\mathfrak{m}^iM}}=\widetilde{\mathfrak{m}^iM_1}\ \ \text{for all} \ i\geq 1.$$ This implies $\widetilde{G(M)}/x_1^*\widetilde{G(M)}=\widetilde{G(M_1)}$. So we get  $\widetilde{G(M)}$ is \CM \ and therefore $G(M)$ is generalised \CM.\\
	Let Ass$_{G(A)}G(M)=\{\mathcal{M},\mathcal{P}_1,\ldots,\mathcal{P}_s\}$, where $\mathcal{M}$ is maximal homogeneous ideal of $G(A)$ and $\mathcal{P}_i$'s are minimal primes in $G(A)$ (see \ref{ASSG}). Set $V=\mathfrak{m}/\mathfrak{m}^2$. We know that $\mathcal{P}_i\cap V\neq V$. Now if dim$\mathcal{P}_i\cap V=$ dim$V-1$, then dim$G(A)/\mathcal{P}_i\leq1$ and this is a contradiction as $\mathcal{P}_i$'s are minimal primes in $G(A).$\\
	Thus, dim$\mathcal{P}_i\cap V\leq$ dim$V-2$. So there exists $u^*,v^*\in V$ such that $H=ku^*+kv^*$ and $H\cap \mathcal{P}_i=0$ for $i=1,\ldots,s$ (see \ref{vector}). Thus if $\xi\in \mathfrak{m}$ such that $\xi^*\in H$ is non-zero then $\xi $ is a superficial element of $M$ (see \cite[Theorem 1.2.3]{Rossi}). Now since $\ell(\widetilde{\mathfrak{m}M}/\mathfrak{m}M)=1$, $$\widetilde{\mathfrak{m}M}=\mathfrak{m}M+Aa\ \text{for some }\ a\not\in \mathfrak{m}M$$
	If $x_1a\in \mathfrak{m}^2M$ then $a\in(\mathfrak{m}^2M:x_1)=\mathfrak{m}M$ and this is a contradiction.\\
	So $\overline{x_1a}\neq0$ in $\widetilde{\mathfrak{m}^2M}/\mathfrak{m}^2M$.\\ Now from the exact sequence (see \ref{RR-2})
	\begin{align}
	0\rt \mathfrak{m}^{2}M:x_1/\mathfrak{m}M\rt \widetilde{\mathfrak{m}M}/\mathfrak{m}M\rt \widetilde{\mathfrak{m}^{2}M}/\mathfrak{m}^{2}M\rt \widetilde{\mathfrak{m}^{2}M_1}/\mathfrak{m}^{2}M_1. 
	\end{align}
	we get $\ell(\widetilde{\mathfrak{m}^2M}/\mathfrak{m}^2M)=\ell(\widetilde{\mathfrak{m}M}/\mathfrak{m}M)=1$. So we have $\overline{ua}=\beta\overline{va}$ where $\beta$ is unit and $u,v$ are $M$-superficial elements. Now we have $(u-\theta v)a\in \mathfrak{m}^2M$ where $\theta\in A$ and $\overline{\theta}=\beta$ is a unit. Since $(u-\theta v)^*=u^*-\beta v^*$ is nonzero element in $H$, so $u-\theta v$ is $M$-superficial. This implies that $$a\in (\mathfrak{m}^2M:(u-\theta v))=\mathfrak{m}M$$
	This is a contradiction. So $\widetilde{\mathfrak{m}M}=\mathfrak{m}M$.\\
	Since $\widetilde{\mathfrak{m}M}=\mathfrak{m}M$, now  from exact sequence (see \ref{RR-2})
	\begin{align}
	0\rt \frac{\mathfrak{m}^{n+1}M:x_1}{\mathfrak{m}^nM}\rt \frac{\widetilde{\mathfrak{m}^nM}}{\mathfrak{m}^nM}\rt \frac{\widetilde{\mathfrak{m}^{n+1}M}}{\mathfrak{m}^{n+1}M}\rt \frac{\widetilde{\mathfrak{m}^{n+1}M_1}}{\mathfrak{m}^{n+1}M_1}, 
	\end{align}
	we get $\widetilde{\mathfrak{m}^iM}=\mathfrak{m}^iM$ for all $i$, because $\widetilde{\mathfrak{m}^iM_1}=\mathfrak{m}^iM_1$ for all $i\geq 2$. This implies that depth$G(M)\geq1.$ Notice that here depth$G(M)=1$ because depth$G(M_1)=0$. In this case $h_M(z) = 3+4z+(1-z)^4$.\\
	Now assume  dim$M\geq5$ and $\underline{x}=x_1,\ldots,x_d$  a maximal $\phi$-superficial sequence. set $M_{d-4}=M/(x_1,\ldots,x_{d-2})M$.  Then we have the following cases.\\
	First case when $G(M_{d-4})$ is \CM. By Sally-descent $G(M)$ is \CM\ and $h_M(z)=4+3z$.\\
	Second case when depth$G(M_{d-4})=3$. By Sally-descent depth$G(M)=d-1$ and
	$h_M(z)=4+2z+z^2$.\\
	Third case when depth$G(M_{d-4})=2$. By Sally-descent depth$G(M)=d-2$ and $h_M(z)=4+z+3z^2-z^3$.\\
	Fourth case when depth$G(M_{d-4})=1$. By Sally-descent depth$G(M)=d-3$ and $h_M(z)=3+4z+(1-z)^4$.

	{\bf Case(5): $e(M)=8$}\\
	In this case  $M_d\cong Q'/(y^2)\oplus Q'/(y^2)\oplus Q'/(y^2)\oplus Q'/(y^2)$ and $h_{M_d}(z)=4+4z$. So $e(M_d)=\mu(M_d)i(M_d)$. For dim$M\geq 1$,  this equality is preserved modulo any $\phi$-superficial sequence. This implies that $G(M)$ is \CM \ (see \cite[Theorem 2]{PuMCM}).
\end{proof}

From the above theorem we can conclude:
\begin{enumerate}
	\item If $e(M)=4 $ then $a_1=a_2=a_3=a_4=1$. In this case $M$ is an Ulrich module so $G(M)$ is \CM\ and $h_M(z)=4.$
	\item If $e(M)=5$ then $a_1=a_2=a_3=1,a_4=2$. In this case we have two cases:
	\begin{enumerate}
		\item  $G(M)$ is \CM\ if and only if  $h_M(z)=4+z$.
		\item  depth$G(M)=d-1$ if and only if $h_M(z)=4+z^2$. 
	\end{enumerate}
	\item If $e(M)=6$ then $a_1=a_2=1,a_3=a_4=2$. In this case we have three cases:
	\begin{enumerate}
		\item  $G(M)$ is \CM\ if and only if $h_M(z)=4+2z$.
		\item  depth$G(M)=d-1$ if and only if $h_M(z)=4+z+z^2$ or $h_M(z)=4+2z^2$.
		\item  depth$G(M)=d-2$ if and only if $h_M(z)=4+3z^2-z^3$.
	\end{enumerate}
	\item If $e(M)=7$ then $a_1=1,a_2=a_3=a_4=2$. In this case we have four cases:
	\begin{enumerate}
		\item $G(M)$ is \CM\ if and only if $h_M(z)=4+3z$.
		\item depth$G(M)=d-1$ if and only if $h_M(z)=4+2z+z^2$.
		\item depth$G(M)=d-2$ if and only if $h_M(z)=4+z+3z^2-z^4$.
		\item depth$G(M)=d-3$ if and only if $h_M(z)=3+4z+(1-z)^4$.
	\end{enumerate}
	\item If $e(M)=8$ then $a_1=a_2=a_3=a_4=2$. In this case $G(M)$ is \CM\ and $h_M(z)=4+4z.$
\end{enumerate}

\begin{corollary}
	Let $(A,\mathfrak{m})$ be a  hypersurface ring of dimension $d$ with $e(A)=3$ and $M$ an MCM module. If  $\mu(M)=4$, then depth$G(M)\geq d-3$.
	
\end{corollary}
\begin{proof}
	By \ref{Base change} we can assume that $A$ is a complete local ring with infinite residue field.
	Since $e(A)=3$,  we can take $(A,\mathfrak{m})=(Q/(g),\mathfrak{n}/(g))$ where $(Q,\mathfrak{n})$ is a regular local ring of dimension $d+1$ and $g\in \mathfrak{n}^3\setminus\mathfrak{n}^4$.\\
	Now we have two cases here.\\
	First case when $M$ has no free summand. In this case, from the above theorem depth$G(M)\geq d-3$.\\
	Next case when $M$ has  free summand. In this case we can write $M\cong N\oplus A^s$ for some $s\geq 1$ and $N$ has no free summand. We assume $N\ne 0$, otherwise  $M$ is free and  $G(M)$ is \CM.\\
	Clearly, $N$ is a MCM $A$-module (see \cite[Proposition 1.2.9]{BH}). Notice that red$(N)\leq 2$, because red$(M)\leq 2$. Also $\mu(M)> \mu(N)$, so $\mu(N)\leq 3$.\\ 
	If $\mu(N)=1$ then we have a minimal presentation of $N$ as $0\rt Q\xrightarrow{a} Q\rt N\rt  0$, where $a\in \mathfrak{n}$. 
	This implies $N\cong Q/(a)Q$. Notice that  since red$(N)\leq 2$, $N$ has no free summand and $i(N)$ is preserved modulo any $\phi$-superficial sequence, this implies $a\in \mathfrak{n}^2\setminus\mathfrak{n}^3$. Since $N\cong Q/(a)Q$, $G(N)$ is \CM.\\
	If $\mu(N)=2$ then depth$G(N)\geq d-1$ (from Theorem \ref{muM=2}).\\	
	If $\mu(N)=3$ then  depth$G(N)\geq d-2$ (from Theorem \ref{muM=3}).\\ 
	We know that (see \cite[Proposition 1.2.9]{BH})
	
	depth$G(M)\geq $ min\{depth$G(N)$, depth$G(A)$\}$=$ depth$G(N)$.\\
	So in this case depth$G(M)\geq d-2$. 
	
\end{proof}

\section{\bf The case when $\mu(M)=r$ and $det(\phi)\in \mathfrak{n}^{r+1}\setminus \mathfrak{n}^{r+2}$}

We konw (from \cite[Theorem 2]{PuMCM}) that if $\mu(M)=r$ and $det(\phi) \in \mathfrak{n}^r\setminus\mathfrak{n}^{r+1}$, then $M$ is an Ulrich module. This implies $G(M)$ is \CM.\\ Here we consider the case when $det(\phi)\in \mathfrak{n}^{r+1}\setminus \mathfrak{n}^{r+2}$. 
\begin{theorem}
	Let $({Q},\mathfrak{n})$ be a regular local ring (with infinite residue field) of dimension $d+1$ with $d\geq 0$. Let $M$ be a $Q$-module with minimal presentation $$0\rt Q^r\xrightarrow{\phi} Q^r \rt M \rt 0$$
	
	Now if $\phi = [a_{ij}]
	$
	where $a_{ij} \in \mathfrak{n}$ with  $f=det(\phi) \in \mathfrak{n}^{r+1}\setminus \mathfrak{n}^{r+2}$ and $red(M)\leq2$, then depth$G(M)\geq d-1$. In this case we can also prove that
	\begin{enumerate}
		\item $G(M)$ is \CM \  if and only if $h_M(z)=r+z$.
		\item depth$G(M)=d-1$ if and only if $h_M(z)=r+z^2$
	\end{enumerate}
\end{theorem}

\begin{proof} Set $(A,\mathfrak{m})=(Q/(f),\mathfrak{n}/(f))$. Since $f.M=0$, so $M$ is an $A$-module. Also, it is clear that $M$ is an MCM $A$-module because projdim$_QM=1$.
	
	We first consider the case when dim$M=2$ because if dim$M\leq1$ there is nothing to prove.\\
	Let   $\underline{x}=x_1,x_2$ be a maximal $\phi$-superficial sequence (see \ref{phi}). Set $M_1=M/x_1M$, $M_2=M/\underline{x}M$ and $(Q',(y))=(Q/(\underline{x}),\mathfrak{n}/(\underline{x}))$.\\
	Clearly, $Q'$ is a DVR.
	
	We know that $det(\phi\otimes Q') = det\phi\in \mathfrak{n}^{r+1}\setminus\mathfrak{n}^{r+2}$ and $\phi$ is an $r\times r$-matrix. So $M_2\cong  Q'/(y)\oplus\ldots\oplus Q'/(y)\oplus Q'/(y^2) $.
	This implies $h_{M_2}(z)=r+z$, where $r=\mu(M).$
	
	Since dim$M_1=1$, we have
	$h_{M_1}(z)=r+(\rho_0-\rho_1)z+\rho_1z^2$
	where $\rho_n=\ell(\mathfrak{m}^{n+1}M_1/{x_2\mathfrak{m}^nM_1})$. As red$M\leq2$ we get $\rho_n=0$ for all $n\geq2$.\\ We also have $\rho_0=\ell(\mathfrak{m}M_1/x_2M_1)=1$ and from \ref{d=1}(1) we know that coefficients of $h_{M_1}$ are non-negative.  So $\rho_1=0$ or $1$.

	{\bf Subcase(i):  $\rho_1=0$.}\\
	In this case, $M_1$ has minimal multiplicity and this implies $G(M_1)$ is \CM\ and $h_{M_1}(z)=r+z$. By Sally-descent $G(M)$ is \CM\ and $h_M(z)=r+z$.

	{\bf Subcase(ii): $\rho_1=1$.}\\ 
	In this case $h_{M_1}(z)=r+z^2$ and depth$G(M_1)=0$ because $h_{M_1}(z)\ne h_{M_2}(z)$ (see \ref{Property}). \\
	From \ref{mod-sup} we can write $$h_M(z)=h_{{M_1}}(z)-(1-z)^2b_{x_1,M}(z).$$
	This gives us $$e_2(M)=e_2({M_1})-\sum b_i(x_1,M),$$
	where $b_i(x_1,M)=\ell(\mathfrak{m}^{i+1}M:x_1/\mathfrak{m}^iM)$.
	We know that $e_2(M)$ and $\sum b_i(x_1,M) $ are non-negative integers. We also have $e_2({M_1})=1$ . So we have  $\sum b_i(x_1,M)\leq 1 $. 
	
	Since  red$M\leq2$, from the exact sequence (\ref{exact seq})
	\begin{align*}
	0 \rt \mathfrak{m}^nM:(\underline{x})/\mathfrak{m}^{n-1}M\rt \mathfrak{m}^nM:x_1/\mathfrak{m}^{n-1}M & \rt \mathfrak{m}^{n+1}M:x_1/\mathfrak{m}^{n}M\\
	\rt \mathfrak{m}^{n+1}M/(\underline{x})\mathfrak{m}^nM
	&\rt \mathfrak{m}^{n+1}{M_1}/x_2\mathfrak{m}^n{M_1}\rt 0
	\end{align*}
	we get if $b_1(x_1,M)=0$ then $b_i(x_1,M)=0$ for all $i\geq2$.\\
	{\bf Claim:} depth$G(M)=1$.\\
	{\bf Proof of the claim:} If possible assume that depth$G(M)=0$. This implies  $\sum b_i(x_1,M) \ne 0$ (see \ref{Property} ).\\
	So we get  $b_1(x_1,M) =1$. 
	
	From the above exact sequence,
	we get 
	$$0\rt \mathfrak{m}^2M:x_1/\mathfrak{m}M\rt\mathfrak{m}^2M/(\underline{x})\mathfrak{m}M\rt\mathfrak{m}^2{M_1}/x_2\mathfrak{m}{M_1}\rt0.$$
	So, $\ell(\mathfrak{m}^2M/(\underline{x})\mathfrak{m}M)=\rho_1+b_1(x_1,M)=2$.

	Consider $\overline{G(M)}=G(M)/(x_1^*,x_2^*)G(M)$. So we have
	$$G(M)/(x_1^*,x_2^*)G(M)=M/\mathfrak{m}M \oplus \mathfrak{m}M/((\underline{x})M+\mathfrak{m}^2M)\oplus \mathfrak{m}^2M/((\underline{x})\mathfrak{m}M+\mathfrak{m}^3M).$$
	Since deg$h_{M_2}=1$, $\mathfrak{m}^2M\subseteq (\underline{x})M$ and  $\mathfrak{m}^3M=(\underline{x})\mathfrak{m}^2M\sub (\underline{x})\mathfrak{m}M$. \\
	Thus we have
	$$\overline{G(M)}=M/\mathfrak{m}M \oplus \mathfrak{m}M/JM \oplus \mathfrak{m}^2M/J\mathfrak{m}M$$
	Its Hilbert series is $r+z+2z^2$. From \ref{overline{G(M)}} this is not a possible Hilbert series. This implies that depth$G(M)\geq 1$. Notice that $G(M)$ cannot be a \CM\ module, because depth$G(M_1)=0$. So depth$G(M)=1$.
	
	Now assume dim$M\geq3$ and $\underline{x}=x_1,\ldots,x_d$ a maximal $\phi$-superficial sequence. Set $M_{d-2}=M/(x_1,\ldots,x_{d-2})M$. We now have two cases.\\
	First case when $G(M_{d-2})$ is \CM. By Sally-descent $G(M)$ is \CM\ and $h_M(z)=r+z$.\\
	Second case when depth $G(M_{d-2})=1$. By Sally-descent depth$G(M)=d-1$ and $h_M(z)=r+z^2$.

\end{proof}

\section{\bf The case when $e(M)=\mu(M)i(M)+1$}

Let $(Q,\mathfrak{n})$  be a regular local ring and
$0\rt Q^{\mu(M)}\xrightarrow{\phi} Q^{\mu(M)}\rt M\rt 0$ be a minimal presentation of $M$ over $Q$. Set $i(M)=$ max\{$i | $ all entries of $\phi $ are in $\mathfrak{n}^i$\}.\\
From \cite[theorem 2]{PuMCM} we know that for an MCM module over a hypersurface ring $e(M)\geq \mu(M)i(M)$ and if $e(M)= \mu(M)i(M)$ then $G(M)$ is \CM. We now consider the next case and prove that:
\begin{theorem}\label{em=mum}
	Let $(Q,\mathfrak{n})$ be a regular local ring of dimension $d+1$, $g\in \mathfrak{n}^i\setminus\mathfrak{n}^{i+1}$, $i\geq 2$. Let $(A,\mathfrak{m})=(Q/(g),\mathfrak{n}/(g))$  and $M$ be a MCM $A$-module. Now  if $e(M)=\mu(M)i(M)+1$ then depth$G(M)\geq d-1$ and $h_M(z)=\mu(M)(1+z+\ldots+z^{i(M)-1})+z^s$ where $s\geq i(M)$. Furthermore, $G(M)$ is \CM\ if and only if $s=i(M)$.  
\end{theorem}
\begin{proof}By \ref{Base change} we can assume that $A$ has infinite residue field.\\
	We first consider the case when dim$M=2$ because if dim$M\leq1$ there is nothing to prove.\\
	Let   $\underline{x}=x_1,x_2$ be a maximal $\phi$-superficial sequence (see \ref{phi}). Set $M_1=M/x_1M$, $M_2=M/\underline{x}M$, $J=(x_1,x_2)$ and $(Q',(y))=(Q/(\underline{x}),\mathfrak{n}/(\underline{x}))$.\\
	Clearly, $Q'$ is a DVR and so $M_2\cong  Q'/(y^{i(M)})\oplus\ldots\oplus Q'/(y^{i(M)})\oplus Q'/(y^{i(M)+1}) $, because $e(M)=\mu(M)i(M)+1$. This implies $h_{M_2}(z)=\mu(M)(1+z+\ldots+z^{i(M)-1})+z^{i(M)}$. Notice that since $\underline{x}$ is $\phi$-superficial sequence, $e(M)=e(M_1)=e(M_2),\mu(M)=\mu(M_1)=\mu(M_2)$ and $i(M)=i(M_1)=i(M_2)$.
	
	Since dim$M_1=1$.  We can write $h$-polynomial of ${M_1}$ as  $h_{M_1}(z)=h_0({M_1})+h_1({M_1})z+\ldots+h_s({M_1})z^s$
	with all the coefficients non-negative (see \ref{d=1}(1)).\\ Now if we set 
	$Q_1=Q/(x_1)$ then $\mathfrak{m}^nM_1/{\mathfrak{m}^{n+1}M_1}\cong \mathfrak{m}^n(Q_1)^{\mu(M)}/{\mathfrak{m}^{n+1}(Q_1)^{\mu(M)}} $ for $n\leq i(M)-1$, because $M_1=coker(\phi\otimes Q_1)$. This implies 
	$$\ell(\mathfrak{m}^nM_1/{\mathfrak{m}^{n+1}M_1})=n\mu(M)\ \text{for all}\ n\leq i(M)-1.$$
	So we have $$h_0(M_1)+h_1(M_1)\ldots+h_n(M_1)=n\mu(M) \text{ for all}\ n\leq i(M)-1.$$
	
	Since $h_0(M_1)=\mu(M)$, we get   $h_0(M_1)=h_1(M_1)=\ldots=h_{i(M)-1}(M_1)=\mu(M)$.\\ Therefore
	$h_{M_1}(z)=\mu(M)(1+z+\ldots+z^{i(M)-1})+z^{s}$ for $s\geq i(M)$, because $e_0(M)=i(M)\mu(M)+1$ and  all its coefficients are non-negative (see \ref{d=1}(1)). \\
	Since $M=coker(\phi)$, for $n\leq i(M)-1$ we get $$\mathfrak{m}^nM/{\mathfrak{m}^{n+1}M}\cong \mathfrak{m}^n(Q)^{\mu(M)}/{\mathfrak{m}^{n+1}(Q)^{\mu(M)}}. $$ 
	
	Now if $h$-polynomial of $M$ is  $h_M(z)=h_0(M)+h_1(M)z+\ldots+h_t(M)z^t$ then 
	$$\ell(\mathfrak{m}^nM/\mathfrak{m}^{n+1}M)=\binom{n+2}{n}\mu(M)\ \text{for all}\  n\leq i(M)-1.$$
	So for all $ n\leq i(M)-1$ we have 
	\[
	(n+1)h_0(M)+nh_1(M)+\ldots +h_n(M)=\binom{n+2}{n}\mu(M). \tag{$\dagger$}
	\]
	Now since $h_0(M)=\mu(M)$, we get from ($\dagger$) $$h_0(M)=h_1(M)=\ldots=h_{i(M)-1}(M)=\mu(M)$$
	So we have, $h_n(M)=h_n({M_1})$ for all $n\leq i(M)-1$.\\
	From Singh's equality (\ref{mod-sup}) we have $$\mathfrak{m}^{n+1}M:x_1=\mathfrak{m}^nM\ \text{for}\ n=0,\ldots,i(M)-1.$$ So we have
	\begin{equation}\label{d211}
	\mathfrak{m}^{n+1}M\cap x_1M=x_1\mathfrak{m}^nM\ \text{for}\ n=0,\ldots,i(M)-1.
	\end{equation}
	Since $h_n(M_1)= h_n(M_2)$ for $n=0,\ldots,i(M)-1$,
	from Singh's equality (\ref{mod-sup})
	\begin{equation}\label{d111}
	\mathfrak{m}^{n+1}{M_1}:x_2=\mathfrak{m}^n{M_1} \ \text{for}\ n=0,\ldots,i(M)-1
	\end{equation}
	Now we have  $\mathfrak{m}^{n+1}M\cap JM=J\mathfrak{m}^nM\ \text{for}\ n=0,\ldots,i(M)-1$.
	In fact, if $\alpha=ax_1+bx_2\in \mathfrak{m}^{n+1}M$. Going modulo $x_1$ we get $\overline{\alpha}=\overline{b}x_2\in \mathfrak{m}^{n+1}{M_1}$. From (\ref{d111}) we have $\overline{b}\in \mathfrak{m}^n{M_1}$. So we can write $b=c+f$, where $c\in \mathfrak{m}^nM$ and $f\in x_1M$. This implies  $\alpha=ax_1+cx_2+fx_2$. Hence $\alpha-cx_2\in \mathfrak{m}^{n+1}M\cap x_1M$. So from (\ref{d211}) we have $\alpha=cx_2+gx_1$ with $c,g\in \mathfrak{m}^nM$. This implies $\alpha\in J\mathfrak{m}^nM$.
	
	So we have
	\begin{equation}\label{VV11}
	vv_j(M)=\ell\Big(\frac{\mathfrak{m}^{n+1}M\cap JM}{J\mathfrak{m}^nM}\Big)=0 \ \text{for}\ i=0,\ldots,i(M)-1
	\end{equation}
	Since $\mathfrak{m}^{i(M)}M:x_1=\mathfrak{m}^{i(M)-1}M$,  we have from \ref{exact seq}
	\begin{equation}\label{v11}
	v_{i(M)-1}=\ell(\mathfrak{m}^{i(M)}M/J\mathfrak{m}^{i(M)-1}M)=\ell(\mathfrak{m}^{i(M)}{M_1}/x_2\mathfrak{m}^{i(M)-1}{M_1})\leq 1
	\end{equation}
	Notice that last inequality in \ref{v11} is clear from the $h$-polynomial of $M_1$.\\
	Now from conditions (\ref{VV11}) and (\ref{v11}), depth$G(M)\geq 1$ (see \cite[Theorem 4.2.1]{Rossi}). \\
	Now assume dim$M\geq3$ and $\underline{x}=x_1,\ldots,x_d$ a maximal $\phi$-superficial sequence. Set $M_{d-2}=M/(x_1,\ldots,x_{d-2})M$. So, depth$G(M_{d-2})\geq 1$.\\
	By Sally-descent  we get depth$G(M)\geq d-1$ and $h_M(z)=\mu(M)(1+z+\ldots+z^{i(M)-1})+z^s$ where $s\geq i(M)$.
\end{proof}

\begin{corollary}
	Let $Q=k[[x_1,\ldots,x_{d+1}]]$.  Let $M$ be a $Q$-module with minimal presentation $0\rt Q^{\mu(M)}\xrightarrow{\phi}Q^{\mu(M)}\rt M\rt 0$. Set $\phi=\sum_{i\geq i(M)}\phi_i$, where $\phi_i$'s are forms of degree $i$. Now if rank($\phi_{i(M)})=\mu(M)-1$ and det$(\phi_{i(M)}+\phi_{i(M)+1})\ne0$ then depth$G(M)\geq d-1$ and
	$h_M(z)=\mu(M)(1+z+\ldots+z^{i(M)-1})+z^s$ where $s\geq i(M)$. Furthermore, $G(M)$ is \CM\ if and only if $s=i(M)$.
	
\end{corollary}
\begin{proof}
	After row and column reduction we can assume that first $\mu(M)-1$ rows and columns of $\phi_{i(M)}$ form an invertible matrix.\\
	We first consider the case when dim$M=2$ because if dim$M\leq1$ there is nothing to prove.\\
	Let   $\underline{x}=x_1,x_2$ be a maximal $\phi$-superficial sequence (see \ref{phi}). Set $M_1=M/x_1M$, $M_2=M/\underline{x}M$ and $(Q',(y))=(Q/(\underline{x}),\mathfrak{n}/(\underline{x}))$.\\
	Clearly, $Q'$ is a DVR. Now set $\phi'= \phi \otimes Q'$. Then rank$\phi'_{i(M)}=\mu(M)-1$ and
	det$(\phi'_{i(M)}+\phi'_{i(M)+1})\ne0$. Therefore we can assume that ("$\sim$" denotes "similar" )
	\[
	\phi \otimes Q' \sim
	\begin{bmatrix}
	y^{i(M)} & & &\\
	& \ddots & &\\
	& & y^{i(M)} &\\
	& &  &y^{i(M)+1}
	\end{bmatrix}
	\]
	So, $M_2\cong  Q'/(y^{i(M)})\oplus\ldots\oplus Q'/(y^{i(M)})\oplus Q'/(y^{i(M)+1}) $.\\ This implies $h_{M_2}(z)=\mu(M)(1+z+\ldots+z^{i(M)-1})+z^{i(M)}$ and  $e(M)=i(M)\mu(M)+1$. \\ Now the result follows from the  Theorem \ref{em=mum}.
\end{proof}

\section{Examples} {\bf Case(1)} If $\mu(M)=2$ then we have

Take $Q=k[[x,y]]$, $\mathfrak{n}=(x,y)$
\begin{enumerate}
	
	\item $\phi = \begin{pmatrix}
	a & b \\
	c & d
	\end{pmatrix} $ with $ad-bc\ne0$ and $a,b,c,d\in \mathfrak{n}\setminus \mathfrak{n}^2$, then $G(M)$ is \CM \ and $h_M(z)=2$.
	\item $\phi = \begin{pmatrix}
	y^{a_1} & 0 \\
	0 & y^{a_2}
	\end{pmatrix} $ where $1\leq a_i\leq 3$ then $G(M)$ is \CM .
	\item $\phi = \begin{pmatrix}
	y^2 & 0 \\
	x^2 & y
	\end{pmatrix} $ then  $G(M)$ is \CM \ and $h_M(z)=2+z.$ Because if we set $e_1=(1,0)^T$ and $e_2=(0,1)^T$ then $M\cong (Q \oplus Q)/\langle y^2e_1+x^2e_2,ye_2 \rangle$. We can easily calculate $\ell(M/\mathfrak{n}M)=2$, $\ell(\mathfrak{n}M/\mathfrak{n}^2M)=3$ and $\ell(\mathfrak{n}^2M/\mathfrak{n}^3M)=3$. Now since $y^3M=0$ and dim$M=1$, we get deg$h_M(z)\leq 2$. From the above calculation it is clear $h_M(z)=2+z$ and this implies $M$ has minimal multiplicity. 
	\item $\phi = \begin{pmatrix}
	y^2 & 0 \\
	x & y
	\end{pmatrix} $ then  depth$G(M)=0$    and $h_M(z)=2+z^2.$ Because if we set $e_1=(1,0)$ and $e_2=(0,1)$ then $M\cong (Q\oplus Q)/\langle y^2e_1+xe_2,ye_2\rangle$. Now it is clear that $\overline{e_2}\in \widetilde{\mathfrak{n}M}\setminus \mathfrak{n}M$. Also, dim$M=1$ and $\rho_2=0$ so, deg$h_M(z)\leq2$. It is also clear that $\ell(M/\mathfrak{n}M)=2$, $\ell(\mathfrak{n}M/\mathfrak{n}^2M)=2$ because $x\overline{e_2}=y^2\overline{e_1}$. Similar calculation gives that $\ell(\mathfrak{n}^2M/\mathfrak{n}^3M)=3$ and $y^3M=0$. So, $h_M(z)=2+z^2$.
	
\end{enumerate}

{\bf Case(2)} Now if $\mu(M)=3$ then examples are:
Take $Q=k[[x,y]]$, $\mathfrak{n}=(x,y)$
\begin{enumerate}
	
	\item $\phi= \begin{pmatrix}
	y^2 & 0 & 0\\
	0 & y & 0\\
	0 & 0 & y
	\end{pmatrix} $ then  $G(M)$ is \CM \ and $h_M(z)=3+z.$
	\item $\phi= \begin{pmatrix}
	y^{a_1} & 0 & 0\\
	0 & y^{a_2} & 0\\
	0 & 0 & y^{a_3}
	\end{pmatrix} $ with $1\leq a_i\leq 3$ then  $G(M)$ is \CM.
	
	\item $\phi= \begin{pmatrix}
	y^2 & 0 & 0\\
	x & y & 0\\
	0 & 0 & y
	\end{pmatrix} $ then  depth$G(M)=0$,  because $\widetilde{\mathfrak{n}M}\neq \mathfrak{n}M,$  
	
	as $\overline{e_2}\in \widetilde{\mathfrak{n}M}\setminus\mathfrak{n}M.$   We can calculate $h-$polynomial (as in case(1.4) above ) and get $h_M(z)=3+z^2.$ Also notice that $y^3\overline{e_i}=0$ for $i=1,2,3$, this implies $M$ is $Q/(y^3)$-module.
	
	\item  $\phi= \begin{pmatrix}
	x & y & 0\\
	x^2 & x^2 & 0\\
	0 & 0 & x^2
	\end{pmatrix} $ then  depth$G(M)=0$,  because $\widetilde{\mathfrak{n}M}\neq \mathfrak{n}M$, 
	
	as $\overline{e_1}\in \widetilde{\mathfrak{n}M}\setminus\mathfrak{n}M.$ We have $x\overline{e_1}= -x^2\overline{e_2}$, $y\overline{e_1}=-x^2\overline{e_2}$ and $x^2\overline{e_3}=0$. From here we get $(x-y)\overline{e_1}=0$, $x^2(x-y)\overline{e_2}=xx^2\overline{e_2}-yx^2\overline{e_2}=-xy\overline{e_1}+xy\overline{e_1}=0$ and $x^2\overline{e_3}=0$. Therefore, $x^2(x-y)\overline{e_i}=0$ for $i=1,2,3$. So, $M$ is $Q/(x^2(x-y))$-module.

	Now for the next three examples take $Q=k[[x,y,z]]$:
	\item   $\phi= \begin{pmatrix}
	x & y & z\\
	x^2 & x^2 & 0\\
	0 & 0 & x^2
	\end{pmatrix} $ then  depth$G(M)=0$, because $\widetilde{\mathfrak{n}M}\neq \mathfrak{n}M$, 
	
	as $\overline{e_1}\in \widetilde{\mathfrak{n}M}\setminus\mathfrak{n}M.$ Here $M$ is $Q/(x^2(x-y))$-module, because $x^2(x-y)\overline{e_i}=0$ for $i=1,2,3.$
	\item  $\phi= \begin{pmatrix}
	x & y & 0\\
	x^2 & x^2 & 0\\
	0 & 0 & x^2
	\end{pmatrix} $ then depth$G(M)=1$, because $z^*$ is $G(M)-$regular and after going modulo $z^*$ we get depth${G({{N}})}=0$, here ${N}=M/zM$. Notice that $\overline{e_1}\in \widetilde{\mathfrak{n}N}\setminus\mathfrak{n}N$.  Here $M$ is $Q/(x^2(x-y))$-module, because $x^2(x-y)\overline{e_i}=0$ for $i=1,2,3.$
	\item  $\phi= \begin{pmatrix}
	x & 0 & 0\\
	0 & x^2 & 0\\
	0 & 0 & x^2
	\end{pmatrix} $ then depth$G(M)=2$, i.e. $G(M)$ is \CM.
\end{enumerate}
{\bf Case(3):} If $\mu(M)=4$ and $e(A)=3$.

\begin{enumerate}
	\item  Take $Q=k[[x,y,z,t]]$, $\mathfrak{n}=(x,y,z,t)$ and  $\phi= \begin{pmatrix}
	x & y & z & t\\
	x^2 & x^2 & 0 &0\\
	0 & 0 & x^2 & 0\\
	0 & 0 & 0  & x^2
	\end{pmatrix} $
	then
	
	depth$G(M)=0$ because $\overline{e_1}$ where $e_1=(1,0,0,0)^T$ is an element of $\widetilde{\mathfrak{n}M}$. We have $x\overline{e_1}=-x^2\overline{e_2}$, $y\overline{e_1}=-x^2\overline{e_2}$, $z\overline{e_1}=-x^2\overline{e_3}$ and $t\overline{e_1}=-x^2\overline{e_4}$. These relations imply that  $x^2(x-y)\overline{e_i}=0$ for $i=1,2,3,4$. So,  $M$ is $Q/(x^2(x-y))$-module.  
	
	\item  $\phi = \begin{pmatrix}
	x & y & z & 0\\
	x^2 & x^2 & 0 &0\\
	0 & 0 & x^2 & 0\\
	0 & 0 & 0  & x^2
	\end{pmatrix} $ 
	then depth$G(M)=1$.\\ Since $t^*$ is $G(M)-$regular and  after going modulo $t^*$, we get depth${G({N})}=0$, here $N=M/tM$.  Notice that $\overline{e_1}\in \widetilde{\mathfrak{n}N}\setminus\mathfrak{n}N$. Since $x^2(x-y)\overline{e_i}=0$ for $i=1,2,3,4$, this implies $M$ is $Q/(x^2(x-y))$-module.
	\item  $\phi= \begin{pmatrix}
	x   & y   & 0   & 0\\
	x^2 & x^2 & 0   &0\\
	0   & 0   & x^2 & 0\\
	0   & 0   & 0   & x^2
	\end{pmatrix} $ then it is clear that $z^*,t^*$ is maximal $G(M)$-regular sequence. So, depth$G(M)=2$. In fact, if we set $N=M/(z,t)M$ then $\overline{e_1}\in \widetilde{\mathfrak{n}M}\setminus\mathfrak{n}M$. Also notice that $M$ is $Q/(x^2(x-y))$-module, because $x^2(x-y)\overline{e_i}=0$ for $i=1,2,3,4.$
	\item  $\phi= \begin{pmatrix}
	x & 0 & 0 & 0\\
	0 & x^2 & 0 &0\\
	0 & 0 & x^2 & 0\\
	0 & 0 & 0  & x^2
	\end{pmatrix} $ then  $G(M)$ is \CM.
	\item Take $Q=k[[x,y,z]]$, $\mathfrak{n}=(x,y,z)$ and  $\phi= \begin{pmatrix}
	x & y & z & 0\\
	x^2 & x^2 & 0 &0\\
	0 & 0 & x^2 & 0\\
	0 & 0 & 0  & x
	\end{pmatrix} $\\
	then depth$G(M)=0$ because $\overline{e_1}$ where $e_1=(1,0,0,0)^T$ is an element of $\widetilde{\mathfrak{n}M}$. Since $x^2(x-y)\overline{e_i}=0$ for $i=1,2,3,4$, this implies $M$ is\\ $Q/(x^2(x-y))-$module.
\end{enumerate}
{\bf Case(4):} If $\mu(M)=r$; take  $Q=k[[x,y]]$, $\mathfrak{n}=(x,y)$ 	

\begin{enumerate}
	\item $[\phi]_{r\times r}= \begin{pmatrix}
	y^2 & 0 &0 &\cdots & 0\\
	x^2 & y & 0 &\cdots & 0\\
	0   & 0 & y &\cdots & 0\\
	\vdots & \vdots &\vdots& \ddots &0\\
	0      &  0     &  0   & \cdots & y
	\end{pmatrix} $ then $det\in\mathfrak{n}^{r+1}\setminus\mathfrak{n}^{r+2}$, $G(M)$ is \CM \ and $h_M(z)=r+z.$
	
	\item  $[\phi]_{r\times r}= \begin{pmatrix}
	y^2 & 0 &0 &\cdots & 0\\
	x & y & 0 &\cdots & 0\\
	0   & 0 & y &\cdots & 0\\
	\vdots & \vdots &\vdots& \ddots &0\\
	0      &  0     &  0   & \cdots & y
	\end{pmatrix} $then $det\in\mathfrak{n}^{r+1}\setminus\mathfrak{n}^{r+2}$, depth$G(M)=0$ because $\widetilde{\mathfrak{n}M}\neq \mathfrak{n}M$ as $\overline{e_2}\in \widetilde{\mathfrak{n}M}\setminus\mathfrak{n}M$.  and $h_M(z)=r+z^2.$

\end{enumerate}

{\bf Case(5):} rank($\phi_{i(M)})=\mu(M)-1$ and det$(\phi_{i(M)}+\phi_{i(M)+1})\ne 0$.
\begin{enumerate}
	\item Take $Q=k[[x,y,z]] $, $\mathfrak{n}=(x,y,z)$ and\\
	 $\phi= \begin{pmatrix}
	x & 0 & 0\\
	0 & x & 0\\
	0 & 0 & x^2
	\end{pmatrix} $ then depth$G(M)=2$, i.e. $G(M)$ is \CM.
	\item  $\phi= \begin{pmatrix}
	x & 0 & y\\
	0 & x & 0\\
	0 & 0 & x^2
	\end{pmatrix} $ then depth$G(M)=1$, because $z^*$ is $G(M)-$regular. If we set $N=M/zM$ then $\overline{e_1}\in \widetilde{\mathfrak{n}M}\setminus\mathfrak{n}M$.
\end{enumerate}

\end{document}